\documentclass{amsart}
\usepackage[utf8]{inputenc}
\usepackage{a4wide}

\usepackage{amsmath,amssymb,amsthm}
\usepackage{comment}
\usepackage{hyperref}
\usepackage{pgf,tikz}
\usepackage{bbm} 

\usetikzlibrary{arrows}
\usetikzlibrary{arrows.meta, calc, quotes}

\newcommand{\RR}{\mathbb R}

\newcommand{\NN}{\mathbb N}

\renewcommand{\SS}{\mathbb S}

\renewcommand{\H}{\mathcal H}
\newcommand{\eps}{\varepsilon}

\newcommand{\loc}{\mathrm{loc}}

\renewcommand{\vec}[1]{\mathbf{#1}}
\newcommand{\abs}[1]{\left\vert #1 \right\vert}
\newcommand{\Abs}[1]{\left\Vert #1 \right\Vert}
\newcommand{\enclose}[1]{\left(#1\right)}
\newcommand{\Enclose}[1]{\left[#1\right]}
\newcommand{\ENCLOSE}[1]{\left\{#1\right\}}

\newcommand{\defeq}{:=}

\renewcommand{\H}{\mathcal{H}}

\newtheorem{theorem}{Theorem}[section]
\newtheorem{proposition}[theorem]{Proposition}

\newtheorem{lemma}[theorem]{Lemma}
\newtheorem{corollary}[theorem]{Corollary}

\theoremstyle{definition}
\newtheorem{definition}[theorem]{Definition}
\theoremstyle{remark}
\newtheorem{remark}[theorem]{Remark}

\title{locally isoperimetric partitions}
\author[Novaga]{M. Novaga}
\author[Paolini]{E. Paolini}
\author[Tortorelli]{V. M. Tortorelli}

\begin{document}
\maketitle

\begin{abstract}
  Locally isoperimetric $N$-partitions are partitions of the space $\RR^d$ into $N$ regions 
  with prescribed, finite or infinite measure, which have minimal perimeter 
  (which is the $(d-1)$-dimensional measure of the interfaces between the regions) among 
  all variations with compact support preserving the total measure of each region.
  In the case when only one region has infinite measure, the problem reduces to the 
  well known problem of isoperimetric clusters: in this case the minimal perimeter is finite, 
  and variations are not required to have compact support.

  In a recent paper by Alama, Bronsard and Vriend, 
  the definition of isoperimetric partition was introduced, and an example,
  namely the \emph{lens} partition, was shown to be locally isoperimetric in the plane.
  In the present paper we are able to give more examples of isoperimetric partitions: 
  in any dimension $d\ge 2$ we have the \emph{lens}, 
  the \emph{peanut} and the \emph{Releaux} triangle. 
  For $d\ge 3$ we also have a \emph{thetrahedral} partition.
  To obtain these results we prove a \emph{closure} theorem which enables us to 
  state that the $L^1_\loc$-limit of a sequence of isoperimetric clusters 
  is an isoperimetric partition, if the limit partition 
  is composed by \emph{flat} interfaces outside a large ball.
  In this way we can make use of the known results about standard clusters. 
  
  In the planar case $d=2$ we have a complete understanding of locally isoperimetric partitions: 
  they exist if and only if the number of regions with infinite area is at most three.
  Moreover if the total number of regions is at most four then, up to isometries, 
  there is a unique locally isoperimetric partition which is the lens, the peanut or the Releaux partition already mentioned.
\end{abstract}

\tableofcontents

\section{Introduction}

The isoperimetric problem is to find a set $E\subset \RR^d$ 
which minimizes its perimeter $P(E)$ 
among all sets with fixed measure $\abs{E} = m$. 
It is well known that the solution to this problem 
(unique up to translation) is the ball of volume $m$.

The isoperimetric problem can be extended to clusters, i.e.\ to 
$N$-uples of disjoint sets $E_1,\dots, E_N$ (regions) which have prescribed 
finite measure $\abs{E_k} = m_k$ and which minimize the total surface area of 
the interfaces. 
Since the common boundary between two regions is only counted once, 
the regions in an isoperimetric cluster are indeed encouraged 
to share part of their boundary. 
This effectively forms a cluster of bubbles.
The existence of isoperimetric clusters with given measures
is guaranteed by the direct method of the calculus of variations \cite{Alm76}.
The problem can be settled in the family of sets with finite perimeter,
also called Caccioppoli sets. 
The perimeter of a Caccioppoli set is indeed lower 
semicontinuous with respect to $L^1$ convergence of sets.
Even if there is no compactness (since $\RR^d$ is unbounded) 
one can still use the tools of concentration compactness to
obtain a minimizer (see, for example, \cite{Mag12, NoPaStTo21}).
Isoperimetric clusters are partially regular in the sense that 
up to a closed \emph{singular} set of zero $(d-1)$-dimensional
measure the boundary of each region 
is a smooth hypersurface with constant mean curvature (see \cite{Alm76, Mag12}). 
In the planar case, $d=2$, the boundaries between two regions  are composed by 
a finite number of circular arcs or straight lines segments, 
the singular set is composed by a finite number of points (vertices), 
and in each vertex exactly three boundaries meet together with equal angles.

In the case $N\le d+1$ for any given $N$-uple of positive measures, 
there is a natural cluster, 
the so called \emph{standard} $N$-bubble (see Definition~\ref{def:standard}), 
which is conjectured to be the unique minimizer, up to isometries,
for the isoperimetric problem.

For $N=2$ (double bubbles) 
isoperimetric clusters have been proven to be standard 
by Foisy et al. \cite{Foi93} when $d=2$, 
and by Hutchinson-Morgan-Ritoré-Ros \cite{HutMorRitRos02} when $d\ge 2$.
For $N=3$ (triple bubbles)
the result (isoperimetric clusters are standard) has been obtained 
by Wichiramala \cite{Wic04} in the case $d=2$, and by the recent results 
of Milman and Neeman \cite{MilNee22} when $d\ge 2$.
For $N=4$ the same result   
has been obtained again in \cite{MilNee22} when $d\ge 4$.
For $N=5$, $d\ge 5$ in \cite{MilNee23} the same authors 
prove that standard clusters are isoperimetric.
They also prove, \cite[Theorem 1.9]{MilNee22}, 
that any isoperimetric $N$-cluster of $\RR^d$, 
has connected regions if $N\le d$.

In the case $N\ge d+2$ there is no notion of standard cluster.
But even in this case it is conjectured that all regions comprising an 
isoperimetric cluster should be connected.
Even in the simplest case $N=4$ and $d=2$ the problem is open 
(but see \cite{PaoTam16, PaoTor20} for the case $d=2$, $N=4$ and equal 
areas).

Inspired by the paper \cite{AlaBroVri23} we are 
going to consider the case when some regions of the cluster 
may have infinite measure.
If $E_1,\dots,E_N$ is a cluster in $\RR^d$ we can add 
an \emph{external} region $E_0$ of infinite measure to 
obtain a \emph{partition} $E_0,E_1,\dots,E_N$ of the whole 
space $\RR^d$. 
An $N$-cluster can thus be considered as an $(N+1)$-partition 
where one of the regions has infinite measure, and the others 
have prescribed finite measure.
If we require two (or more) regions to have infinite measure then the 
partition cannot have finite perimeter and hence 
it makes no sense to minimize the perimeter of the whole partition. 
We are thus led to consider a different, local notion of minimizer as 
follows: 
a partition is said to be \emph{locally isoperimetric}
if for every bounded set $\Omega$ the perimeter of the partition 
in the interior of $\Omega$ 
is minimal among all partitions composed by regions with 
the same prescribed measures whose difference 
(set theoretic symmetric difference) with 
the corresponding region of the original partition is compactly contained 
in $\Omega$ (see Definition~\ref{def:isoperimetric} below).

\begin{figure}
  \begin{tikzpicture}[line cap=round,line join=round,>=triangle 45,x=0.5cm,y=0.5cm]
    \clip(-3,-0.1) rectangle (3,2.1);
    \draw [domain=1.7320508075688772:8.680000000000005] plot(\x,{(--3.41-0*\x)/3.41});
    \draw [domain=-7.860000000000002:-1.7320508075688774] plot(\x,{(-1.41-0*\x)/-1.41});
    \draw [shift={(0,0)}] plot[domain=0.52:2.62,variable=\t]({1*2*cos(\t r)+0*2*sin(\t r)},{0*2*cos(\t r)+1*2*sin(\t r)});
    \draw [shift={(0,2)}] plot[domain=3.67:5.76,variable=\t]({1*2*cos(\t r)+0*2*sin(\t r)},{0*2*cos(\t r)+1*2*sin(\t r)});
  \end{tikzpicture}\hfill%
  \begin{tikzpicture}[line cap=round,line join=round,>=triangle 45,x=0.5cm,y=0.5cm]
    \clip(-3,-0.1) rectangle (4,2.1);
    \draw [shift={(3.86,1)}] plot[domain=2.95:3.34,variable=\t]({1*2.59*cos(\t r)+0*2.59*sin(\t r)},{0*2.59*cos(\t r)+1*2.59*sin(\t r)});
    \draw [shift={(1.96,-0.39)}] plot[domain=0.77:1.9,variable=\t]({1*2*cos(\t r)+0*2*sin(\t r)},{0*2*cos(\t r)+1*2*sin(\t r)});
    \draw [shift={(1.96,2.39)}] plot[domain=4.38:5.52,variable=\t]({1*2*cos(\t r)+0*2*sin(\t r)},{0*2*cos(\t r)+1*2*sin(\t r)});
    \draw [shift={(0,0)}] plot[domain=0.85:2.62,variable=\t]({1*2*cos(\t r)+0*2*sin(\t r)},{0*2*cos(\t r)+1*2*sin(\t r)});
    \draw [shift={(0,2)}] plot[domain=3.67:5.43,variable=\t]({1*2*cos(\t r)+0*2*sin(\t r)},{0*2*cos(\t r)+1*2*sin(\t r)});
    \draw [domain=-4.0119959030658:-1.7320508075688772] plot(\x,{(-0.88-0*\x)/-0.88});
    \draw [domain=3.400216203626535:5.093753580545594] plot(\x,{(--0.82-0*\x)/0.82});
    \end{tikzpicture}\hfill%
  \begin{tikzpicture}[line cap=round,line join=round,>=triangle 45,x=0.5cm,y=0.5cm]
    \clip(-4,-2.5) rectangle (3.5,1);
    \draw [shift={(-0.78,-1.33)}] plot[domain=-0.01:1.04,variable=\t]({1*1.54*cos(\t r)+0*1.54*sin(\t r)},{0*1.54*cos(\t r)+1*1.54*sin(\t r)});
    \draw [shift={(0.76,-1.34)}] plot[domain=2.08:3.13,variable=\t]({1*1.54*cos(\t r)+0*1.54*sin(\t r)},{0*1.54*cos(\t r)+1*1.54*sin(\t r)});
    \draw [shift={(0,0)}] plot[domain=4.18:5.23,variable=\t]({1*1.54*cos(\t r)+0*1.54*sin(\t r)},{0*1.54*cos(\t r)+1*1.54*sin(\t r)});
    \draw [domain=0.7560532952001398:3.6790608748351423] plot(\x,{(-0.84-0.56*\x)/0.95});
    \draw [domain=-4.020930204499883:-0.7830200017915531] plot(\x,{(--0.49-0.31*\x)/-0.55});
    \draw [domain=0.0:3.6790608748351423] plot(\x,{(-0--0.83*\x)/0.01});
    \end{tikzpicture}%
    \caption{The lens, the peanut and the Releaux partition.}
    \label{fig:lente}
    \label{fig:peanut1}
    \label{fig:releaux}
\end{figure}
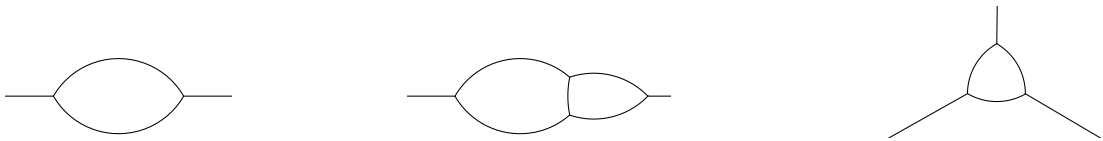

In \cite{AlaBroVri23} the authors consider the planar case $d=2$
with $\abs{E_0}=\abs{E_1} = \infty > \abs{E_2} > 0$,
and introduce the concept 
of locally isoperimetric partitions.
Moreover they show that the \emph{lens} partition,
whose boundary is composed by two symmetric arcs and two half lines
joining with equal angles, 
(see Figure~\ref{fig:lente}) is, in fact, locally isoperimetric. 

In this paper we investigate the properties of locally isoperimetric 
partition. 
First of all we aim to prove a closure theorem. 
The idea is that isoperimetric clusters are a particular case of
partitions with a single infinite region. 
But any partition, even with two or more infinite regions, can be seen 
as the limit (in the sense of $L^1_\loc$ topology) of a sequence 
of clusters. 
We are not currently able to show that the limit of any sequence of isoperimetric 
clusters is indeed a locally isoperimetric partition: 
the problem is that the local volume constraint of the finite regions which converge to 
an infinite region is lost in the limit. 
However we are able to prove a closure result (see Theorem~\ref{th:closure})
which guarantees that such a limit is isoperimetric with respect to variations which 
preserve the volumes not only of the finite regions but also of the infinite ones 
(since the variations have bounded support one can put a constraint on the volume in 
a large ball, taking into account that outside the ball the region coincides 
with its variation).
Next we show that if the limit partition is \emph{flat} at infinity 
(see Definition~\ref{def:eventually_flat}) then we are able to prove that
the constraint on the infinite regions can be dropped because flat regions 
can be modified by a variation which allows us to change the volume by any 
large amount with a small change of the perimeter 
(see Theorem~\ref{th:large_volumes}).
These results allow us to find many examples of isoperimetric partitions in $\RR^d$. 
The lens partition in $\RR^d$ is the limit of double bubbles hence 
we can state that it is locally isoperimetric in any dimension.
By taking the limit of a triple bubble we obtain the \emph{peanut} partition, which is 
composed by two finite and two infinite regions, in any dimension.
Again, by considering the limit of triple bubbles we obtain a partition with three infinite 
regions and only one finite region that we call \emph{Releaux partition} since 
in dimension $d=2$ the finite region is a Releaux triangle.
See Figure~\ref{fig:lente}.
The other cases covered by the results of \cite{MilNee22}
are the partitions obtained as limits of quadruple bubbles in $\RR^d$ for $d\ge 4$.

In the planar case we are able to find some more results.
First of all we show that all locally isoperimetric partitions can have only 
one, two or three infinite regions. 
Then we prove that if we assign any number of areas for the finite 
regions, and we have from one to three infinite regions, a locally 
isoperimetric partition, with the given areas, always exists (Theorem~\ref{th:existence}).
This is achieved by taking minimal clusters with the prescribed areas inside 
a ball whose radius is going to infinity. 
The difficult point in this approach is that of proving that no mass escapes to infinity:
it is here that we need the assumption $d=2$ and use some tools already developed for 
planar clusters.
In the cases of the examples above, namely the lens (two infinite and one finite region), 
the peanut (two infinite and two finite regions), the Releaux triangle (three infinite and one 
finite region) we can show that these example, up to isometries, are in fact the unique 
locally isoperimetric partitions with their prescribed areas 
(Theorem~\ref{th:uniqueness}).

In the Appendix
we sketch how to adapt some statements on isoperimetric clusters, 
\emph{e.g.} the so called Almgren's Lemma and the Infiltration Lemma,
to limits of locally isoperimetric partitions, 
to get volume density estimates and a priori boundedness of their  
regions with finite volume (see also Theorem~\ref{th:recall}).

To summarize, in this paper we give a clear understanding of what locally isoperimetric 
partitions in the plane look like. We have also many examples in higher dimension, which rely on 
the corresponding results for clusters. 
These results are obtained by means of a closure 
theorem which is valid in any dimension
while in the planar case we adapt to infinite regions the tools already 
developed in the study of planar clusters.

There are many questions left open. Of course any question which is open for the case of clusters 
is open in the case of partitions. For example we do not know, in general, if the bounded regions 
of an isoperimetric partition are connected. 
Also, in the case $d\ge 3$, when we prove that standard partitions are locally isoperimetric
we are not able to prove that they are the unique example (up to isometries) with their volume.

A particular case of isoperimetric partitions is the case when all regions 
have infinite measure. 
In this case, since there are no volume constraints, 
isoperimetric partitions can be called \emph{locally minimal partitions}.
If $N=2$ and $d\le 7$
it is known that all locally minimal partitions are half-spaces 
(see \cite{Sim68,Deg65})
For $N=2$ and $d\ge 8$ there are also minimal cones
which are not hyperplanes, for example Simons cone (see \cite{BomDeGGiu69}) 
and Lawson cones (see \cite{Law72}).
These are examples of locally isoperimetric partitions which are not standard.
For $N=2$ and $d\ge 9$ it is possible to find locally minimal partitions 
which are not cones (this is the Bernstein problem, see \cite{BomDeGGiu69}).
For any $N$ in $d=3$ all locally minimal conical partitions have been classified 
(see \cite{Tay73}): apart from the half-spaces there are 
only the dihedral angle with 120 degrees and the tetrahedral cone.
Of course a cylinder over a locally minimal partition is also locally minimal, so we can 
produce many examples also in higher dimension, but a complete classification is still missing
in the case $d>3$.
There are also a few more examples of minimal cones with $N>2$ and $d>3$: 
for $N=d+1$ the cone over the skeleton
of the symplex (see \cite{LawMor94}) and for $N=2d$ the cone over the skeleton of the hypercube 
(see \cite{Bra91}).

Notice that in $d\ge 8$ the Simons and Lawson cones are locally isoperimetric partitions 
which are not eventually flat (see Definition~\ref{def:eventually_flat}).
However, it seems reasonable that the flatness condition that we have used in our examples 
could be replaced with the condition of having zero mean curvature, 
with some decay at infinity of the second fundamental form.
In this respect we could expect to find a locally isoperimetric partition 
in $\RR^8$ with prescribed volumes $(1,\infty,\infty)$ which is not eventually flat 
and hence is different from the lens partition with the same volumes.

\smallskip
{\bf Acknowledgements}
  The authors wish to thank Emanuel Milman for useful comments on a preliminary version of this work.
  The first author was partially supported by the project PRIN 2022 GEPSO, 
  financed by the European Union – NextGenerationEU, and 
  the second author was supported by the Project PRIN 2022PJ9EFL.
  The first and second authors are members of the INDAM-GNAMPA.

\section{Notation and preliminary results}

We shall denote by $\omega_d$ the volume of the unit ball in $\RR^d$.
A set $E\subset \RR^d$ is said to be a \emph{Caccioppoli set}
or \emph{set with locally finite perimeter}
if $E$ is measurable and the distributional derivative ${D\mathbbm{1}_E}$
of its characteristic function is a Radon vector measure.
We define $P(E,B)$, the perimeter of $E$ in $B$
by means of the total variation of the characteristic function: $\Abs{D\mathbbm{1}_E}(B)$. 
We denote $P(E,\RR^d)$ by $P(E)$ .

We identify Caccioppoli sets which differ by a set of measure zero.
In particular we can always choose a suitable representant of $E$ 
so that the topological boundary coincides with the 
measure theoretic boundary (see for example \cite[proposition 12.19]{Mag12}):
\[
  \partial E = \ENCLOSE{x\in \RR^d\colon \forall \rho>0 \colon 0<\abs{E\cap B_\rho(x)}<\abs{B_\rho}}.
\]
Notice that $\partial E$ is a closed set.
For a general measurable set $E$ one can define a \emph{reduced boundary}
$\partial^* E\subset \partial E$ comprising the points 
where the approximate outer normal unit vector 
$\nu_E$ can be defined so that $D\mathbbm{1}_E = \nu_E \cdot \Abs{D\mathbbm{1}_E}$. 
When $E$ is a Caccioppoli set one has $P(E,B) = \H^{d-1}\enclose{B\cup\partial^* E}$, 
for every borel set $B$. 

\begin{definition}[partition]
Let $\vec E = (E_1,\dots, E_{N})$ be an $N$-uple of measurable subsets of $\RR^d$, 
and $\Omega$ an open set.
We say that $\vec E$ is an $N$-\emph{partition} (of $\Omega$), with {\it regions} 
$E_1,\dots, E_{N}$, if $\abs{E_k\cap E_j\cap \Omega}=0$ for every $k\neq j$
and $\abs{\Omega\setminus \bigcup E_k}=0$. 
\end{definition}
The \emph{boundary} $\partial \vec E$ of a partition is the set of all interfaces 
between the regions:
\[
  \partial \vec E \defeq \bigcup_{k=1}^N \partial E_k.
\]

We define the \emph{perimeter} of $\vec E$ on any borel set $B\subset \RR^d$ 
as
\[
  P(\vec E,B) \defeq \frac 1 2 \sum_{k=1}^{N} P(E_k,B),\qquad P(\vec E) = P(\vec E, \RR^d).
\]
This quantity represents the $(d-1)$-dimensional surface area of the interfaces 
between the regions $E_k$ inside $B$.
In fact when $\vec E$ is sufficiently regular 
(namely when $\H^{d-1}(\partial E_k\setminus\partial^*E_k)=0$ for all $k$) 
we have $P(\vec E,B) = \H^{d-1}(\partial \vec E\cap B)$.

Notice that we do not require the regions of a partition to have positive measure.
If some of the regions have zero measure we say that the partition is \emph{improper},
otherwise we say that it is \emph{proper}. 
Improper partitions are useful to describe the limit of a sequence of proper partitions 
when the measure of some region goes to zero. 
Handling of improper regions is one of the difficulties we had to face in 
Theorem~\ref{th:volume_fixing_variations_alternative} and
Theorem~\ref{th:closure}.

\begin{definition}[locally isoperimetric partition]\label{def:isoperimetric}
We say that an  $N$-partition $\vec E=(E_1,\dots,E_{N})$ of $\Omega\subset \RR^d$ is a 
\emph{locally isoperimetric partition} if for every compact 
set $B\subset \Omega$ given any partition $\vec F=(F_1,\dots, F_{N})$
such that for all $k$ one has $\abs{(E_k\triangle F_k)\setminus B}=0$ and
$\abs{E_k} = \abs{F_k}$ one has
\[
P(\vec E,B) \leq P(\vec F_k,B).
\]
\end{definition}

If $\vec E = (E_0,E_1,\dots, E_N)$ is 
an $(N+1)$-partition of $\RR^d$
such that $\abs{E_k}<+\infty$ for all $k\neq 0$
we say that $\vec E$ is an 
$N$-\emph{cluster} in $\RR^d$.
The sets $E_k$ are called regions of the cluster $\vec E$.
The region $E_0$ necessarily has infinite measure ad hence 
is always unbounded. 
It is usually called the \emph{exterior region}.
The cluster is determined by the $N$-uple $(E_1,\dots,E_N)$
of the finite regions
(and this is the usual definition in the literature)
because $E_0$ is uniquely determined as the complement 
of the union of the other regions.
We choose to add $E_0$ to the definition so that $N$-clusters can be regarded 
as a particular case of $(N+1)$-partitions.

Notice that if a partition has at least two regions with infinite measure 
(we will say: \emph{infinite} regions) then its perimeter is always $+\infty$. 
Clusters, instead, can have finite perimeter in the whole space, 
hence it is natural to give a \emph{global} notion of minimizer for clusters.

\begin{definition}[isoperimetric cluster]
We say that the $N$-cluster $\vec E=(E_0,\dots,E_N)$ is \emph{isoperimetric}
if for every other cluster $\vec F=(F_0,\dots,F_N)$ 
with $\abs{F_k} = \abs{E_k}$ for all $k\neq 0$ one has
\[
  P(\vec E) \le P(\vec F).
\]
\end{definition}

Clearly every isoperimetric $N$-cluster $\vec E$ is also a locally isoperimetric  
$(N+1)$-partition. 
The converse statement is also true but requires a proof, see Proposition~\ref{prop:equivalence}.

In the following theorem, we summarize some regularity results for locally isoperimetric partitions, 
that can be proven with minor modifications from the analogue statements for isoperimetric clusters 
(see the Appendix for a sketch of some of the proofs).

\begin{theorem}
\label{th:recall}
Let $\vec E=(E_1,\dots,E_{N})$ be a locally isoperimetric $N$-partition in $\RR^d$. 
Then $\partial \vec E$ is smooth outside a closed \emph{singular} set of Hausdorff dimension at most $d-2$: 
each interface is an analytic hypersurface with locally constant mean curvature. 
The mean curvature is zero between two infinite regions.

Moreover the following density estimates hold:
\[
  c_0\omega_d r^d\le |E_k\cap B(x,r)|\le c_1 \omega_d r^d
\]
for all $x\in\partial \vec E$ and $r<r_0$,
with $c_0=c_0(d,N)$, $c_1= c_1(d,N)$ and $r_0 = r_0(\vec E)$.

In particular, if $\abs{E_k}<+\infty$ then $E_k$ is bounded.

\end{theorem}

 \begin{lemma}[upper estimate of perimeter]   \label{lm:stima1}
 Let $\vec E=(E_1,\dots,E_{N})$ be a locally isoperimetric $N$-partition in $\RR^d$.
 Then, for every $\vec x\in \RR^d$ and any $R>0$,
 one has
 \[
   P(\vec E, B_R(\vec x)) \le C_0 \cdot R^{d-1}
 \]
 with $C_0=C_0(d,N)$ a constant not depending on $\vec E$, $\vec x$ or $R$.
 \end{lemma}

 \begin{proof}
To simplify the notation suppose $\vec x=0$.
For every $\rho <R$
we can rearrange the regions of $\vec E$ inside 
 the ball $B_R$ into horizontal slices 
 to construct a partition 
 $\vec F=(F_1,\dots,F_{N})$
 such that for all $k$ one has
 \begin{enumerate}
   \item $F_k \setminus B_\rho = E_k\setminus B_\rho$,
   \item $\abs{F_k\cap B_\rho} = \abs{E_k\cap B_\rho}$,
   \item $\partial^*(F_k\cap B_\rho) \subset \partial B_\rho \cup \Pi_k \cup \Pi_k'$
   where $\Pi_k$ and $\Pi_k'$ are parallel $(d-1)$-dimensional planes.
 \end{enumerate}
 Since $\vec E$ is locally isoperimetric
 we have $P(\vec E,B_R)\le P(\vec F,B_R)$ 
 hence 
 \begin{align*}
   2P(\vec E, B_R)
  &\le 2P(\vec E,B_R\setminus \overline{B_\rho})
   + 2\H^{d-1}(\partial B_\rho)
   + 2\sum_{k=1}^{N} \H^{d-1}((\Pi_k\cup\Pi_k')\cap B_\rho)
   \\
  &\le 2P(\vec E,B_R\setminus \overline{B_\rho})
   + 2d\omega_d \rho^{d-1}
  + 4N \omega_{d-1} \rho^{d-1}.
\end{align*}
 As $\rho \to R^-$ we have $P(\vec E,B_R\setminus \overline{B_\rho})\to 0$
 and the result follows.
 \end{proof}
 
The following lemma uses the coarea formula to estimate the cost in perimeter
when two Caccioppoli sets are joined together along the surface of a sphere.
\begin{lemma}[glueing]\label{lm:glueing}
  Let $E$ and $F$ be Caccioppoli sets in $\RR^d$.
  Define 
  \[
     G_\rho = (E\cap B_\rho) \cup (F\setminus B_\rho).
  \]
  Then for all $0<r<R$ one has 
  \begin{align*}
  \int_r^R P(G_\rho, \partial B_\rho )\, d\rho
  & =\int_r^R \Enclose{P(G_\rho,B_R) - P(E,B_\rho) - P(F,B_R\setminus B_\rho)}\, d\rho\\
  & = \abs{(E \triangle F)\cap (B_R\setminus B_r)},
  \end{align*}
  so that the set of $\rho \in (r,R)$ 
  for which the estimate
  \[ 
    P( G_\rho, \partial B_\rho) \le \frac{\vert (E\triangle F) \cap (B_R\setminus B_r)\vert}{R-r}
  \]
  holds, has positive measure.
\end{lemma}
\begin{proof}
  For every $\rho$ by additivity of measures one has $P(G_\rho,\partial B_\rho)=
  P(G_\rho, B_R) -P(E, B_\rho )-P(F,B_R\setminus \overline{B_\rho})$. 
  Since $E$ and $F$ are Caccioppoli sets for almost all $\rho\in [r,R]$ we have 
  \[ 
    P(E, \partial B_\rho )= P(F, \partial B_\rho)=0.
  \]
  In particular $P(F,B_R\setminus \overline{B_\rho})=P(F,B_R\setminus {B_\rho})$.
  Moreover for all these $\rho$, 
  denoting by $E^1$ and $F^1$ the points of density $1$ of $E$ and 
  $F$ respectively, we have (see for example \cite[Chapter 15]{Mag12}):
  \begin{align*}
  P(G_\rho, \partial B_\rho) 
  &= \H^{d-1} (\partial^* G_\rho\cap \partial B_\rho)\\
  &=
  \H^{d-1}(((\partial^*E)\cap B_\rho\cap\partial B_\rho)\triangle
  ((\partial^*F)\setminus B_\rho\cap\partial B_\rho))=
  \\
  &=\H^{d-1}((\partial^*(E\cap B_\rho)\cap\partial B_\rho)\triangle
  (\partial^*(F\setminus B_\rho)\cap\partial B_\rho))
  =\\
  &= \H^{d-1} \left((E^1\cap \partial B_\rho) \triangle (F^1\cap \partial B_\rho) \right)
  \end{align*}
  Integrating in $d\rho$ on the interval $[r,R]$, 
  and using coarea formula, we get the desired equality.
\end{proof}
     
\begin{lemma}[perimeter estimate]
\label{lm:stima2}
Let $\vec E=(E_0,\dots, E_{N})$ be a locally isoperimetric 
$(N+1)$-partition in $\RR^d$ which is an $N$-cluster, i.e.\ 
$m = \abs{E_1}+\dots+\abs{E_{N}}<+\infty$.
Then 
\[
  P(\vec E) \le C < +\infty  
\]
where $C=C(d,N,m)$ does not depend on $\vec E$.
\end{lemma}

\begin{proof}
Let us fix $\eps>0$.
Since $\abs{E_k}<+\infty$ for $k \neq 0$,
we can find a radius ${R} \geq \frac 1 \eps$ 
so large that $\abs{E_k\setminus B_R}<\eps/N$. 
Further enlarging $R$, we can also assume that $B_R$ compactly 
contains $N$ pairwise disjoint balls of volumes $\vert E_k\vert$, for $k=1,\dots, N$.
Since the partition has locally finite perimeter one has that 
$\H^{d-1}(\partial B_\rho\cap \partial^* E_k)=0$ 
for all $k\geq 0$.
and almost all $\rho>0$.
Among these $\rho>0$, by Lemma \ref{lm:glueing} 
(applied with $\emptyset$ in place of $E$, 
with $E_k$ in place of $F$, and $[R,R+1]$ in place of $[r,R]$) 
we can find $\rho=\rho(\eps) \in (R,R+1)$ such that
$P(E_k\setminus {B_\rho},B_{R+1})-P(E_k,B_{R+1}\setminus {B_R})=P(E_k\setminus B_\rho,\partial B_\rho)
<\eps/N$, for all $k\geq 1$. 
Hence we can consider a new partition $\vec F=(F_0,\dots, F_{N})$
such that: 
\begin{enumerate}
  \item $F_k \setminus B_\rho = E_k\setminus B_\rho$, for all $k=0,\dots, N$,
  \item $\abs{F_k} = \abs{E_k}$, for all $k=0,\dots N$,
  \item for $k\neq 0$ then the intersections $F_k \cap B_\rho$ are pairwise disjoint balls
    compactly contained in $B_\rho$ 
    with measure $\abs{F_k\cap B_\rho}=\abs{E_k\cap B_\rho}$, 
  \item $\displaystyle F_0= \RR^d\setminus \bigcup_{k=1}^N F_k$,
\end{enumerate}
Note that since $F_k$ are Caccioppoli sets, 
and $\partial B_\rho$ can be viewed as the intersection of a decreasing sequence of 
open sets on which $P(\vec F, \cdot)$ is finite,
one has 
\[
  P(F_0,\partial B_\rho)=P\left(\bigcup_{k=1}^NF_k,\partial B_\rho\right)\le \sum_{k=1}^N P(F_k, \partial B_\rho ).
\]
Now, since $\vec E$ is locally isoperimetric, 
and taking in account that $\overline{B_{R}}$ 
is the intersection of a decreasing sequence of open sets 
on which $P(F_k,\cdot)$ is finite, we have
\begin{align*}
P(\vec E,\overline{B_{\rho}})
  & \le P(\vec F,\overline{B_{\rho}})
  =\frac 12 \sum_{k=0}^N P(F_k, \overline{B_{\rho}})\\
  &=\frac 12 \sum_{k=0}^N P(F_k,{B_{\rho}})+\frac 12 \sum_{k=0}^N P(F_k,\partial{B_{\rho}})\\ 
  &=C_d\sum_{k=1}\vert E_k\cap B_\rho\vert^{\frac {d-1} d}+\frac 12 \sum_{k=0}^N P(F_k,\partial{B_{\rho}}) 
  \le C_dNm^{\frac{d-1}d} +{\eps}.
\end{align*}
Hence, since $\rho>R\geq \frac1 \eps$, letting $\eps\to 0^+$, we get 
\[ 
  P(\vec E) \le C_dNm^{\frac{d-1}d}.
\] 
\end{proof}

\begin{theorem}[volume fixing, possibly improper, variations]
\label{th:volume_fixing_variations_alternative}  
  Let $\Omega$ be an open subset of $\RR^d$.
  Let $\vec F = (F_1,\dots, F_N)$ be a (possibly improper) $N$-partition 
  of $\Omega$.
  For all $k=1,\dots,N$ with
  $\abs{F_k}>0$, let $x_k\in \Omega$, and $\rho_k>0$, be given so that
  the balls $B_{\rho_k}(x_k)$ are contained in $\Omega$, 
  are pairwise disjoint
  and $\abs{F_k\cap B_{r}(x_k)}>\frac 1 2 \omega_d r^d$
  for all $r\le \rho_k$.
 
  Let $A=\bigcup_k B_{\rho_k}(x_k)$.
  Then for every $\vec a = (a_0,\dots, a_N)$
  with $a_k \ge -\abs{F_k\cap B_{\rho_k}(x_k)}$ when $\abs{F_k}>0$,
  $a_k\ge 0$ when $\abs{F_k}=0$, 
  and $\sum_{k=1}^N a_k = 0$,
  there exists a partition $\vec F' = (F'_1, \dots, F'_N)$
  such that for all $k=1,\dots, N$:
  \begin{enumerate}
    \item $F_k' \triangle F_k \subset A$,
    \item $\abs{F'_k \cap A} = \abs{F_k \cap A} + a_k$,
    \item $P(F'_k, A) \le P(F_k,A) + C_1 \cdot 
      \displaystyle\sum_{j=1}^N \abs{a_j}^{1-\frac 1 d}$,
  \end{enumerate}
  with $C_1=C_1(d,N)$ not depending on $\vec F$.  
\end{theorem}
\begin{remark}
 Notice that it is always possible to find such balls $B_{\rho_k}(x_k)$: 
 since $\abs{F_k}>0$ just 
  take any  point $x_k$ of full density in $F_k$, 
  and $\rho_k$ sufficiently small.
\end{remark}
\begin{remark}The previous theorem can be compared to the so called \emph{Volume fixing variations} theorem leading to \emph{Almgren's Lemma}
(see Appendix \ref{th:volume_fixing_variations}, \ref{lm:almg},  \cite[Theorem 29.14, Corollary 29.17]{Mag12}) with two important differences.
First of all we do not require the regions to have positive measure. 
This enables us to make a \emph{blow-down} with regions having measure 
going to zero (see the proof of Theorem~\ref{th:uniqueness}).
On the other hand our estimate (3) is weaker than in the usual volume fixing variations theorem,  
since in (3) we have an exponent $1 - \frac 1 d$ instead of $1$. 
If $\abs{F_k}=0$ the exponent $1-\frac 1 d$ is optimal in view of the isoperimetric inequality. 
Otherwise one could prove that the exponent 
$1-\frac 1 d$ can be in fact replaced with $1$ (this, however, 
requires a longer and more refined proof which we prefer to avoid here).
\end{remark} 

\begin{proof}[Proof of Theorem~\ref{th:volume_fixing_variations_alternative}]
  Let $J=\ENCLOSE{j\colon a_j<0}$.
For all $j\in J$ take $r_j$ such that $\abs{F_j\cap B_{r_j}(x_j)} = -a_j$.
By assumption $0 < -a_j \le \abs{F_j\cap B_{\rho_j}(x_j)}$ and hence 
$r_j$ exists and $0 < r_j\le \rho_j$.
Coinsider the sets $F_k'' = F_k\setminus \bigcup_{j\in J} B_{r_j}(x_j)$.
Clearly we have $\abs{F_k''\cap A} \le \abs{F_k\cap A}+a_k$ for all $k=1,\dots,N$ 
so that we are now required to 
add measure to each region. 
Since $\sum a_k = 0$ the total measure we need to add,
which is $\sum_k \abs{F_k\cap A} - \sum_k \abs{F_k''\cap A}$ is exactly equal to the total 
measure of the balls we have removed $\sum_{j\in J} \abs{B_{r_j}(x_j)}$. 
This means that it is possible to find a partition 
$\vec C = (C_1,\dots, C_N)$ of $\bigcup_{j\in J} B_{r_j}(x_j)$ with 
$\abs{C_k} = \abs{F_k\cap A} - \abs{F_k''\cap A} + a_k \ge 0$. 
So we consider the partition $\vec F'$ with regions $F_k' = F_k'' \cup C_k$ to obtain 
the desired volumes:
\[
  \abs{F_k'\cap A} = \abs{F_k\cap A} + a_k.  
\]
By construction $F_k'\triangle F_k \subset A$ for all $k=1,\dots,N$.
To estimate the perimeter we observe that if we choose $\vec C$ 
by making slices of the balls with parallel planes 
for each $j\in J$, we are adding at most $N$ slices, and also 
at most the perimeter of the ball. So the increase in perimeter is at most:
\[
 P(\vec F',\bar A)  - P(\vec F,\bar A) 
  \le \sum_{j\in J} \Enclose{N\omega_{d-1} r_k^{d-1} + d\cdot \omega_d r_k^{d-1}}.
\]
Since, by assumption, $\abs{a_j} =\abs{F_j\cap B_{r_j}(x_j)} \ge \frac 1 2 \omega_d r_j^d$
we have 
 \[
    r_j^{d-1} \le \frac{2}{\omega_d}\abs{a_j}^{1-\frac 1 d}
 \]
hence, as desired,
\[
  P(\vec F',\bar A)  - P(\vec F,\bar A) 
  \le \sum_{j\in J} 2\cdot \frac{N\omega_{d-1} + d\cdot \omega_d}{\omega_d} \abs{a_j}^{1-\frac 1 d}
  \le C_1 \sum_{j=1}^N \abs{a_j}^{1-\frac 1 d}.
\]
\end{proof}

\begin{proposition}[equivalence of isoperimetric clusters and locally isoperimetric partitions]
  \label{prop:equivalence}
  If $\vec E=(E_0, E_1, \dots, E_N)$ is an $N$-cluster in $\RR^d$ then 
  $\vec E$ is an isoperimetric $N$-cluster 
  if and only if $\vec E$ is a locally isoperimetric $(N+1)$-partition.
\end{proposition}
By using Almgren's Lemma (see Theorem~\ref{th:recall}, Lemma~\ref{lm:almg}, Corollary~\ref{cor:bnd})
one could prove that the regions with finite measure 
of a locally isoperimetric partition are, in fact, bounded.
This would make the following proof much easier, however are able 
to present a self contained proof which uses 
Theorem~\ref{th:volume_fixing_variations_alternative} instead of the classical one,  
adapted to partitions (see Appendix Theorem~\ref{th:volume_fixing_variations}, 
and Lemma~\ref{lm:almg}).
\begin{proof}[Proof of Proposition~\ref{prop:equivalence}]
  Notice that any competitor in the definition of a locally isoperimetric partition
  is also a competitor in the definition of an isoperimetric cluster,
  where we drop the requirement of the variation to have compact support.
  Hence it is clear that an isoperimetric cluster 
  is an isoperimetric partition.

  On the other hand let $\vec E=(E_0,\dots,E_N)$ be a locally isoperimetric partition 
  with $\abs{E_k}<+\infty$ for $k\neq 0$
  and let 
  $\vec F=(F_0,\dots,F_N)$ be a \emph{global} variation i.e.\ a partition
  such that $\abs{F_k} = \abs{E_k}$ for $k\neq 0$ (necessarily $\abs{F_0}=\abs{E_0}=+\infty$).
  To prove that $\vec E$ is an isoperimetric cluster it is enough to show that 
  given any $\eps>0$ we have 
  \[
      P(\vec E) \le P(\vec F) + 2\eps.
  \]
  
  By Lemma~\ref{lm:stima2} we know that each $E_k$ has finite perimeter.
  Suppose also $P(\vec F)<+\infty$ (otherwise there is nothing to prove).
  
  Consider a large radius $\tilde R$ so that 
  \[
    \sum_{k=1}^N \abs{F_k\setminus B_R} + \abs{E_k\setminus B_R}< \eps, \qquad \sum_{k=0}^N P(F_k\setminus B_R) +P(E_k\setminus B_R) < \eps.
  \]
  
  and define
  \[
    F'_k = (F_k\cap B_R) \cup (E_k\setminus B_R).
  \]
  Since for $k\neq 0$ the region $E_k$ has finite measure and
  the complementary of $E_0$ has also finite measure, 
  using Lemma~\ref{lm:glueing} in an interval $[\tilde R,\tilde R+\delta]$
  we choose $R>0$ large enough we get 
  \begin{align*}
    P(\vec F') &\le 
      P(\vec F,B_R) + P(\vec E,\RR^d\setminus B_R) + 2\eps \le P(\vec F)+ 4\eps.
  \end{align*}
  
  Now applying Theorem~\ref{th:volume_fixing_variations_alternative} we can slightly modify 
  $\vec F'$ inside $B_R$ to obtain a partition $\vec G=(G_0,\dots,G_N)$ such that 
  $\abs{G_k} = \abs{E_k}$ and $G_k\triangle E_k$ is bounded for all $k=0,\dots, N$. 
  Hence we can finally state that $P(\vec E) \le P(\vec G)$. Whence 
  \[
    P(\vec E) 
      \le P(\vec G) 
      \le P(\vec F')+\eps
      \le P(\vec F)+4\eps.
  \]
\end{proof}
  
\begin{definition}[isoperimetric partition with mixed constraint]
  \label{def:mixed_isoperimetric}
  Let $J\subset \ENCLOSE{1,\dots,N}$ be a fixed set of indices.
  We say that an $N$-partition $\vec E = (E_1, \dots, E_N)$ of an 
  open set $\Omega$
  is locally $J$-isoperimetric,
  if, whenever we are given a compact set $B\subset \Omega$ and a partition
  $\vec F=(F_1,\dots, F_N)$ of $\Omega$ such that 
  $F_i \triangle E_i\subset B$ for all $i=1,\dots, N$,
  and $\abs{F_j\cap B} = \abs{E_j\cap B}$ for all $j\in J$, then, 
  we have,
  \[
     P(\vec E,B) \le P(\vec F, B).
  \]
\end{definition}

\begin{theorem}[closure for $J$-isoperimetric partitions] 
  \label{th:closure}
  Let $J$ be a subset of $\ENCLOSE{1,\dots,N}$
  and let $\vec E^k = (E^k_1,\dots, E^k_N)$
  be a sequence of locally $J$-isoperimetric (possibly improper)
  $N$-partitions of $\Omega_k \subset \RR^d$
  where $\Omega_k$ is an increasing sequence of open sets
  such that $\bigcup_k \Omega_k = \RR^d$. 
  Suppose that there exists $\vec E=(E_1,\dots, E_N)$, 
  a partition of $\RR^d$,
  such that
  for all $j=1,\dots, N$,
  we have $E^k_j\to E_j$ in $L^1_{\loc}(\RR^d)$ 
  as $k\to +\infty$.
  
  Then $\vec E = (E_1,\dots, E_N)$ is a locally $J$-isoperimetric partition
  of $\RR^d$.
\end{theorem}
\begin{remark}
Since given any $K\subset \RR^d$, compact set, one has 
  $\Omega_k\supset K$ for $k$ large enough, the $L^1_{\loc}$ convergence
  makes sense in this setting.

\end{remark}
\begin{proof}
Let $\vec F= (F_1,\dots, F_N)$
be a competitor to $\vec E$ i.e.\ an $N$-partition of $\RR^d$ such that
$F_i \triangle E_i\Subset \RR^d$ for all $i=1,\dots, N$ and 
$\abs{F_j\setminus E_j} = \abs{E_j\setminus F_j}$ for all $j\in J$. 
For all $j$ such that $\abs{F_j}>0$ we can take any point $\vec x_j\in\RR^d$ with 
full density in $F_j$. 
Then take $\rho_j$ such that the hypotesis 
of Theorem~\ref{th:volume_fixing_variations_alternative} are satisfied.
Define 
\[
  \mu = \frac 1 N \cdot \min_{\abs{F_j}>0} \abs{F_j\cap B_{\rho_j}(\vec x_j)}.
\]
Let $A$ be the union of the balls $B_{\rho_j}(\vec x_j)$ given by 
the Theorem.
Let $r>0$ be large enough so that $A\Subset B_r$ and $E_i\triangle F_i \Subset B_r$ 
for all $i=0,\dots, N$.
By slightly enlarging $r$ we can also assume that 
$P(\vec E,\partial B_r)=P(\vec F,\partial B_r)=0$ and hence
\begin{equation}\label{eq:12435687}
   P(\vec E,\bar B_r) = P(\vec E, B_r), 
   \qquad 
   P(\vec F,\bar B_r) = P(\vec F, B_r).
\end{equation}
Let $\eps>0$ be given. By possibly decreasing $\rho_j$ we can assume that 

\begin{equation}\label{eq:78653421}
    C_1 \cdot N^2 \cdot \mu^{1-\frac 1d} \le \eps
\end{equation}
where $C_1$ is the constant given by Theorem~\ref{th:volume_fixing_variations_alternative}.

To conclude the proof it is enough to prove that
\[ 
  P(\vec E, B_r) \le P(\vec F, B_r) + 4\eps. 
\]

First choose $\delta>0$, so that thanks \eqref{eq:12435687}
\begin{equation}\label{eq:3874892}
  P(\vec E, B_{r+\delta})< P(\vec E,\bar B_r)+\eps = P(\vec E,B_r) + \eps.
\end{equation}
Then we can choose $k\in \NN$ sufficiently large so that
$B_{r+\delta} \Subset \Omega_k$. 
Let 
\[
  m_i = \abs{(E^k_i\triangle E_i)\cap B_{r+\delta}}.
\]
By the $L^1_{\loc}$ convergence of $\vec E^k$ to $\vec E$,
by taking $k$ sufficiently large,
we might also assume that for all $i=1,\dots, N$
one has
\begin{align}\label{eq:3278544}
  m_i \le  \mu,\qquad
  m_i  \le \frac{\delta\cdot \eps}{N}.
\end{align}
Using the semicontinuity of perimeter
we can finally also assume $k$ so large that
\begin{equation}\label{eq:384643}
P(\vec E, B_r) \le P(\vec E^k, B_r)+\eps.
\end{equation}

Take now $\rho\in (r , r+\delta)$ and consider
the $N$-partition $\vec F^k = (F^k_1, \dots, F^k_N)$ of $\Omega_k$ 
defined by
\[ 
  F^k_i = (F_i\cap B_\rho) \cup (E^k_i\setminus B_\rho), \qquad i=1,\dots,N
\]
so that $F^k_i$ is a variation of $E^k_i$
with compact support in $\bar B_\rho$.

By a suitable choice of $\rho$ in the interval $(r,r+\delta)$, 
thanks to Lemma~\ref{lm:glueing} and \eqref{eq:3278544}, 
we are not spending too much perimeter:
\begin{equation}\label{eq:1346643}
  \begin{aligned}
  \abs{P(\vec F^k,\bar B_\rho) - P(\vec F,\bar B_\rho)}
  &=P(\vec F^k, \partial B_\rho)
  \le  \frac{1}{\delta}\sum_{i=1}^N \abs{(E_i\triangle E_i^k)\cap B_{r+\delta}}\\
  &= \frac{1}{\delta}\sum_{i=1}^N m_i
  \le \eps . 
  \end{aligned}
\end{equation}

To have a competitor to the minimality of $\vec E^k$ we need 
to slightly modify the partition $\vec F^k$ in $\Omega_k$  to satisfy the 
mixed volume constraint.
To achieve this we want to apply Theorem~\ref{th:volume_fixing_variations_alternative}, 
modify the partition $\vec F$ inside of 
$A\Subset B_r\Subset  B_\rho\Subset B_{r+\delta}\Subset \Omega_k$.
Consider the sets of indices $K=\ENCLOSE{j\not\in J\colon \abs{F_j}=0}$ 
and $L=\ENCLOSE{j\not \in J\colon \abs{F_j}>0}$ so that 
$\ENCLOSE{1,\dots,N} = J\cup K\cup L$.
Then define 
\[
   a_j = \begin{cases}
    \abs{E^k_j\cap B_\rho} - \abs{F^k_j\cap B_\rho} 
      &\text{if $j\in J$}\\
    c &\text{if $j\in L$}\\
    d &\text{if $j\in K$}
   \end{cases}  
\]
where $c$ and $d$ are defined by taking
\[
    S = \sum_{j\in J} a_j, 
    \qquad 
    c = \begin{cases}
      \text{irrelevant} & \text{if $L=\emptyset$}\\
      -\frac{S}{\#L} & \text{if $L\neq \emptyset$}.
    \end{cases},
    \qquad 
    d = \begin{cases}
      \text{irrelevant} & \text{if $K=\emptyset$}\\
      -\frac{S}{\#K} & \text{if $L=\emptyset$}\\
      0 & \text{if $L\neq\emptyset$}.\\
    \end{cases}
\]
We claim that $\sum a_j = 0$, 
as required by Theorem~\ref{th:volume_fixing_variations_alternative}.
Notice that since both $E^k_j$ and $F^k_j$ cover the whole ball $B_\rho$
we have 
\[
  \sum_{j=1}^N \abs{E_k^j\cap B_\rho} - \abs{F^k_j\cap B_\rho} = 0.
\]
So, if $K=\emptyset$ and $L=\emptyset$ the claim is verified.
If $L\neq \emptyset$ then $d=0$ and, by definition, 
\[
 \sum_{j=1}^N a_j = S + \#L \cdot c + \#K\cdot d = S - S + 0 = 0.  
\]
Otherwise, if $L=\emptyset$ then
\[
  \sum_{j=1}^N a_j = S + \#K \cdot d = S - S = 0.
\]

Now we want to prove that 
$a_j\ge -\abs{F_j\cap B_{\rho_j}(\vec x_j)}$ when $\abs{F_j}>0$ while 
$a_j\ge 0$ when $\abs{F_j}=0$,
as required by Theorem~\ref{th:volume_fixing_variations_alternative}.

If $j\in J$ we have $F^k_j\cap B_\rho = F_j\cap B_\rho$ and 
$\abs{F_j\cap B_\rho} = \abs{E_j\cap B_\rho}$ 
hence $\abs{a_j} = \abs{\abs{E^k_j\cap B_\rho} - \abs{E_j\cap B_\rho}}
\le \abs{(E^k_j\triangle E_j)\cap B_\rho} = m_j\le \mu$, by \eqref{eq:3278544}.
If $\abs{F_j}>0$ we have $\mu \le \abs{F_j\cap B_{\rho_j}(\vec x_j)}$, by definition,  and hence 
$a_j\ge -\mu \ge -\abs{F_j\cap B_{\rho_j}(\vec x_j)}$.
If instead $\abs{F_j}=0$ just notice that $a_j=\abs{E^k_j\cap B_\rho} \ge 0$ while $-\abs{F_j\cap B_{\rho_j}(\vec x_j)}\le 0$.

If $\ell\in L$ we have $\abs{ F_\ell}>0$ and $a_\ell = c$. 
Hence 
\[
  \abs{a_\ell} = \abs{c}= \frac{\abs{S}}{\#L} \le \sum_{j\in J} \abs{a_j} 
  \le N\cdot \mu \le \abs{F_\ell\cap B_{{\rho}_\ell}(\vec x_\ell)}
\]
by defintion of $\mu$.

If $h\in K$ we have $a_h = d$ and $\abs{F_h}=0$. So to satisfy the hypothesis of Theorem~\ref{th:volume_fixing_variations_alternative} we have to prove that $a_h=d\geq 0$. 
If $L\neq \emptyset$ by definition we have $a_h=d=0$ and the conclusion is trivial.
If instead $L=\emptyset$, \emph{i.e.} $\abs{F_i}>0\Rightarrow i\in J$, by definition $a_h=d=-\frac S{\#K}$, and, since we are dealing with partitions, it follows:
\[
  \sum_{j\in J} \abs{E^k_j \cap B_\rho}
  \le \sum_{i=1}^N \abs{E^k_i \cap B_\rho}
  = \sum_{i=1}^N \abs{F_i \cap B_\rho}
  =\sum_{j\in J} \abs{F_j \cap B_\rho},
\]
so that 
\[
  -a_h \cdot \#K
  =S  
  =\sum_{j\in J} \abs{E^k_j \cap B_\rho}-\abs{F_j \cap B_\rho}\leq 0.
\]
Notice that in particular we have 
$\abs{a_j} \le N\cdot \mu$, 
for all  $j=1, \dots, N$.

We are now in the position to apply Theorem~\ref{th:volume_fixing_variations_alternative} to $\vec F^k$ in $A\Subset B_\rho$,
getting a partition $\vec G^k$ of $\Omega_k$, such that
\begin{align}
  G^k_j \triangle F_j 
  &= G^k_j \triangle F^k_j
  \subset A \Subset B_\rho, && \text{for all $j=1,\dots, N$}, \\
\label{eq:8326975}
  \abs{G^k_j\cap B_\rho} 
  &= \abs{E^k_j\cap B_\rho}, && \text{for all $j\in J$,}
\end{align}
and, using \eqref{eq:78653421},
\begin{equation}
\label{eq:8376874} 
  P(\vec G^k,\bar B_\rho)-P(\vec F^k,\bar B_\rho) 
      \le C_1\cdot \sum_{j=1}^N \abs{a_j}^{1-\frac 1 d}
      \le C_1\cdot N^2\cdot \mu^{1-\frac 1 d} \le \eps.
\end{equation}


We eventually obtain that $\vec G^k$ is a competitor to $\vec E^k$ 
with the correct mixed volume constraint. 
Hence by local minimality of $\vec E^k$ one has:
\begin{equation}\label{eq:3492387}
    P(\vec E^k, \bar B_\rho) \le P(\vec G^k, \bar B_\rho).
\end{equation}
%
So the proof is concluded,
using \eqref{eq:12435687},
and letting $\delta\to 0^+$ in the following inequality:
\begin{align*}
  P(\vec E, B_{r+\delta})
  & \le P(\vec E, B_r) + \eps && \text{by  \eqref{eq:3874892} }\\
  & \le P(\vec E^k, B_r) + 2\eps && \text{by \eqref{eq:384643}}\\
  & \le P(\vec E^k, \bar B_\rho) + 2\eps  && \text{}\\
  & \le P(\vec G^k, \bar B_\rho) + 2\eps && \text{by \eqref{eq:3492387}}\\
 & \le P(\vec F^k, \bar B_\rho) + 3\eps && \text{by \eqref{eq:8376874}}\\ 
  & \le P(\vec F, \bar B_\rho)   + 4\eps && \text{by \eqref{eq:1346643}}\\
  & \le P(\vec F, B_{r+\delta}) + 4\eps.
\end{align*}
\end{proof}

\begin{definition}[eventually flat partitions]
  \label{def:eventually_flat}
We say that an $N$-partition $\vec E=(E_1,\dots,$ $E_{N})$ 
of $\RR^d$ is \emph{eventually flat}
if for every pair $i\neq j$ 
of indices such that $E_i$ and $E_j$ have infinite measure
there exists a $(d-1)$-dimensional half space 
contained in the interface $\partial E_i \cap \partial E_j$.
\end{definition}

\begin{theorem}[volume fixing of large volumes]
  \label{th:large_volumes}
Let $\vec E = (E_1,\dots,E_N)$ be an eventually flat partition of $\RR^d$
and let $\eps>0$ and $\vec a = (a_1,\dots,a_N)$ be given.
Suppose that $\sum a_k=0$ and $a_k=0$ if $\abs{E_k}$ is finite. 
Then for every $r>0$, there exists $R>r$ and a partition $\vec F$ 
of $\RR^d$ such that for all $k=1,\dots, N$ one has
\begin{align*} 
  E_k\triangle F_k & \Subset B_R \setminus \bar B_r\\ 
  \abs{F_k \cap B_R} &= \abs{E_k \cap B_R} + a_k\\
  P(\vec F, B_R) &\le P(\vec E, B_R) + \eps.
\end{align*}
\end{theorem}
\begin{proof}
  By assumptions on $a_k$ we only need to fix the volumes of the regions
  with infinite volume.
  Since the partition is assumed to be eventually flat, such regions have interfaces 
  which contain arbitrarily large flat $(d-1)$-dimensional disks.
  To fix the volumes we can simply add or remove a cylinder of very large radius 
  and very small height with basis on such disks. 
  This enables us to obtain arbitrarily large changes in volumes with arbitrarily small 
  change in perimeter.
\end{proof}

\begin{theorem}[closure for locally isoperimetric partitions]
  \label{th:closure_flat}
  Let $\vec E^k = (E^k_1,\dots, E^k_N)$ be a sequence 
  of locally isoperimetric partitions.
  Suppose that there exists $\vec E=(E_1,\dots, E_N)$ 
  a partition of $\RR^d$ 
  such that
  $E^k_i\to E_i$ in $L^1_{\loc}(\RR^d)$ 
  and $\abs{E^k_i}\to \abs{E_i}$
  whenever $\abs{E_i}<+\infty$.
  
  If $\vec E$ is eventually flat, 
  then $\vec E$ is itself a locally isoperimetric partition.
\end{theorem}
\begin{proof}
Let $J=\ENCLOSE{j\colon \abs{E_j}<+\infty}$.
For all $j\in J$,
since $\abs{E^k_j}\to\abs{E_j}$, also $\abs{E^k_j}<+\infty$ 
for $k$ large enough.
This means that $\vec E^k$ is in particular locally $J$-isoperimetric.
By Theorem~\ref{th:closure} we hence obtain that 
also $\vec E$ is locally $J$-isoperimetric.
Let $\vec F = (F_0,\dots, F_N)$ be a competitor to $\vec E$ in the sense of local isoperimetricity.
This means that $F_j\triangle E_j$ are bounded and that 
$\abs{F_j}=\abs{E_j}$ for all $j=0,\dots, N$.
In particular $\vec F$ is eventually flat as $\vec E$.
Let $r>0$ be so large that outside $B_r$ the two partitions $\vec E$ and $\vec F$ coincide.
Define, for all $j=1,\dots, N$,
\[
  a_j = \abs{E_j\cap B_r} - \abs{F_j\cap B_r}.   
\]
If $E_j$ has finite measure, since $E_j \triangle F_j \subset B_r$,
we have $a_j=0$.
So for any $\eps>0$ we can apply Theorem~\ref{th:large_volumes} to obtain 
a partition $\vec G$ which differs from $\vec F$ only inside a 
larger ball $B_R$, which agrees with $\vec F$ inside $B_r$ 
and such that 
\[
  P(\vec G,B_R) \le P(\vec F,B_R) + \eps.  
\]
Now we have $\abs{G_j\cap B_R} = \abs{E_j\cap B_R}$ for all $j=1,\dots,N$ 
hence, by the $J$-isoperimetricity of $\vec E$ we conclude that 
\[
P(\vec E,B_R) \le P(\vec G,B_R) \le P(\vec F, B_R) + \eps
\]
and since $\vec E$ ane $\vec F$ agree outside $B_r$ we conclude
\[
P(\vec E,B_r) \le P(\vec F,B_r) + \eps.  
\]
\end{proof}

\section{Standard isoperimetric partitions}
\label{sec:standard}

\begin{definition}\label{def:standard} (See \cite{MilNee22}). 
We say that a partition $\vec E = (E_1,\dots, E_N)$ of $\RR^d$ 
is \emph{standard} if it can be obtained as \emph{any stereographic 
projection} of an 
\emph{equal-volume standard $(N-1)$-bubble} in $\SS^d$, 
i.e.{} a partition of the sphere $\SS^d$ which is 
the \emph{Voronoi partition} corresponding to $N$ equidistant points 
in $\SS^d$ as a subset of $\RR^{d+1}$.

If only one of the regions of a standard partition 
 has infinite measure we 
 say that the partition is a \emph{standard cluster} or standard bubble.

We call standard $N$-partition, or $(N-1)$-bubble, of $\SS^d$, any stereographic
projection  of an $(N-1)$-cluster of $\RR^d$.
\end{definition}

\begin{remark}
Standard $N$-partitions only exist for $N \le d+2$.
For $N\le d+2$, 
$(N-1)$-standard clusters of $\RR^d$
and $N$-standard partitions of $\SS^d$, 
are unique, up to isometries, if the volumes 
of the regions have been fixed (see \cite{Ami01,MilNee22}).
It is conjectured that all isoperimetric $(N-1)$-clusters in $\RR^d$
(recall that a $(N-1)$-cluster is an $N$-partition) are 
standard when $N\le d+2$. 
Each region of a standard partition shares a boundary with 
every other region. 
\end{remark}

\begin{lemma}[approximation of a standard partition by standard clusters]
\label{lem:approximation}
Let $\vec E=(E_1,\dots, E_N)$ be a standard $N$-partition in $\RR^d$. 
Then there exists a sequence $\vec E^k = (E^k_0,\dots, E^k_N)$
of standard $(N-1)$-clusters which converge in $L^1_\loc$ to $\vec E$, 
and $\abs{E^k_i}\to \abs{E_i}$.
\end{lemma}

\begin{proof}
Let $\vec F$ be the Voronoi partition of $\SS^d$ which corresponds 
to $\vec E$ by means of the stereographic projection. 
If $\vec E$ is itself a cluster then $\partial \vec F$ 
does not contain the north pole of $\SS^d$. 
Otherwise with an arbitrarily small rotation of $\vec F$ on $\SS^d$ we 
obtain a partition $\vec F'$ such that 
$\partial \vec F'$ does not contain the north pole, 
and it belongs to a fixed region $F'_j$.
The corresponding stereographic projection $\vec E'$ will be a cluster 
in $\RR^d$ and when the rotation converges to the identity we obtain 
$L^1$ convergence of the partitions $\vec F'\to \vec F$ on the sphere 
and $L^1_\loc$ convergence of their stereographic projections 
$\vec E'\to \vec E$ in $\RR^d$. 
Clearly if $\partial E_i$ does not contain the north pole 
we have $L^1$ convergence in a ball containing $E_i$ hence $\abs{E'_i}\to \abs{E_i}$.
Otherwise $\abs{E_i}=+\infty$ and $\abs{E_i'}\to +\infty$.
\end{proof}

\begin{corollary}[examples of locally isoperimetric partitions]
  \label{cor:examples}
If $\vec E = (E_1,\dots, E_N)$ is a standard $N$-partition of $\RR^d$ 
and if we know that all standard $(N-1)$-clusters of $\RR^d$ are isoperimetric,
then $\vec E$ is locally isoperimetric.
\end{corollary}
\begin{proof}
If the partition $\vec E$ is a cluster 
then the result follows from Proposition~\ref{prop:equivalence}. 
Otherwise notice that the partition is eventually flat, 
because it is the stereographic projection in $\RR^d$ of a standard partition on the 
sphere $\SS^d$ which has the north pole on the boundary. 
Each of the interfaces joining at the north pole are contained in 
maximal $(d-1)$-spheres in 
$\SS^d$ so that their stereographic projection is contained in a $(d-1)$-dimensional 
plane in $\RR^d$. 
Moreover, each region of a standard partition shares a boundary with 
every other region. 
So Definition~\ref{def:eventually_flat} is satisfied.
The conclusion follows from Lemma~\ref{lem:approximation}, 
Theorem~\ref{th:closure_flat}, and Proposition~\ref{prop:equivalence}.
\end{proof}

By the results on standard clusters already mentioned 
in the introduction,
the above corollary assures that any standard $N$-partition 
of $\RR^d$ is locally isoperimetric
for $N\le \min\{ 5, d+1\}$ \cite{MilNee22} 
or $N=4$ and $d=2$ \cite{Wic04}.
This enables us to give a lot of examples of locally isoperimetric
partitions.

For $N=2$ we have that \emph{half-spaces} are locally isoperimetric partitions in every $\RR^d$.
These can be obtained as the limit of a ball with volume going to infinity.
It is well known that the ball solves the isoperimetric problem.

For $N=3$ we have the \emph{lens} partitions which is the partition of $\RR^d$ composed 
by two half-spaces and a lens-shaped region between them. The lens is composed by two 
symmetrical $(d-1)$-dimensional sphere caps joining in a $(d-2)$-dimensional sphere lying in the 
plane containing the interface between the two unbounded regions.
This partition can be obtained as the limit of double bubbles with a bubble converging to the 
lens and the other converging to an unbounded region.
The isoperimetricity of double bubbles has been proven in 
\cite{Foi93} for $d=2$,
in \cite{HutMorRitRos02} for $d=3$ and in \cite{Rei08} for all dimensions. 
This partition was already shown to be isoperimetric in \cite{AlaBroVri23}.

Again for $N=3$ and $d=2$ we can have the \emph{triple junction} partition 
of $\RR^2$ composed by three unbounded regions whose boundary 
is the union of three half-lines joining with equal angles in a single 
point. 
If $d>2$ we obtain a cylinder over the triple junction partition of $\RR^2$.
These partitions can be seen, again, as the limit of a double bubble in $\RR^d$
by letting the two volumes go to infinity.

For $N=4$ we have the \emph{peanut} partition which is 
a partition with two bounded and two unbounded 
regions obtained by merging together two 
lens partitions with a common planar interface. 
The two lenses can have different volumes. 
This partition can be obtained as the limit of a triple bubble with two bubbles converging 
to the two bounded regions and the third bubble converging to an unbounded region.
In the planar case, this partition has been described in \cite{AlaBroVri23}
and in fact was conjectured to be locally isoperimetric.

Again for $N=4$ we have the \emph{Relaux} partition, 
composed by one bounded and three unbounded regions. 
It is obtained by adding three spherical slices to a triple junction partition.
If $d=2$ the bounded region has the shape of a Relaux triangle. 
In $d=3$ it has the shape of a Brazil nut.
This partition can be obtained as the limit of a triple bubble with one bubble 
converging to the bounded region and the other two (symmetrical)  
bubbles converging to unbounded regions. 

For $N=4$ and $d=3$ we can have the \emph{tetrahedral} partition 
which is a cone-like partition obtained by considering any regular tetrahedron 
and taking as regions each of the four cones with vertex at the center of the tetrahedron
generated by the four faces of the tetrahedron itself.
This partition, which is standard, is the blow up of a triple bubble in $\RR^3$ 
centered in one of the two points in common to all the four regions.
For $d>3$ we obtain a cylinder over the tetrahedral partition of $\RR^3$.

In the case $d\ge 3$
we don't know if these example are unique (up to isometries) 
with their prescribed volumes.
In fact, also for $N\le \min\{ 5, d+1\}$, we cannot exclude that, 
there exists a locally isoperimetric partition 
$\vec F$ of $\RR^d$ which is not the limit of standard clusters.
In that case we would have two different locally isoperimetric partitions, 
the standard one and a non-standard one,
with the same prescribed volumes.
In the case $d=2$ we have instead a uniqueness result, Theorem~\ref{th:uniqueness}.

The main result of \cite{MilNee22} also gives examples of $5$-partitions in $\RR^d$
with $d\ge 4$ which are locally isoperimetric, we do not try to describe 
their geometry.

\section{The planar case}

The following theorem resumes well known properties of minimizers. 
See for example \cite{Mor94,Mag12,Alm76}.

\begin{theorem}[regularity of planar local minimizers]\label{th:recall2}
Let $\vec E=(E_1,\dots,E_N)$ be a locally isoperimetric partition of an open set $\Omega\subset \RR^2$ 
i.e.{}
a partition such that 
for any other given partition $\vec F$ of $\Omega$ with 
$\abs{F_k\cap \Omega}=\abs{E_k\cap \Omega}$ and $F_k\triangle E_k\subset B$ 
for some open bounded $B\Subset \Omega$, one has
\[
  P(\vec E, B) \le P(\vec F, B).
\]
Then the following properties hold:
\begin{enumerate}
  \item $\partial \vec E$ is a locally finite graph composed by straight segments
  or circular arcs meeting in triples with equal angles of 120 degrees;
  \item the three signed curvatures of the arcs meeting in a vertex 
  have zero sum;
  \item it is possible to define a \emph{pressure} $p_i$ for each 
  $i=1,\dots, N$ such that
  the curvature of an arc separating the regions $E_i$ and $E_j$ 
  has curvature $p_i-p_j$ (the sign is chosen so that the curvature is positive 
  when the arc has the concavity towards $E_i$)
\end{enumerate}
If $\vec E$ is any partition satisfying the above properties
we say that $\vec E$ is \emph{stationary}.
If $\vec E(t)=(E_1(t),\dots,E_N(t))$ is a one-parameter curve of partitions 
of $\Omega$
such that $\vec E(t_0)$ is stationary, and 
$E_k(t)\setminus B = E_k(t_0)\setminus B$ for some open set $B\subset \Omega$,
then one has 
\begin{equation}
  \Enclose{\frac{d}{dt} P(\vec E(t),B)}_{t=t_0} 
  = \sum_{k=1}^N p_k \cdot \Enclose{\frac{d}{dt} \abs{E_k(t)\cap B}}_{t=t_0}
\end{equation}
where $p_1,\dots,p_N$ are the pressures of the regions of $\vec E(t_0)$.
\end{theorem}

\begin{theorem}\label{th:boundedness}
Let $\vec E=(E_1,\dots,E_{N})$ be any locally isoperimetric $N$-partition of the plane $\RR^2$.
Then $\partial \vec E$ is connected and the number of regions of $\vec E$ with infinite area is at least 
$1$ and at most $3$.
If only one area is infinite then $\vec E$ is a bounded cluster.
If two areas are infinite then $\partial \vec E$ coincides with 
a straight line outside a sufficiently large ball. 
If three areas are infinite then $\partial \vec E$ coincides, 
outside a sufficiently large ball, with three half-lines whose 
prolongations define angles of 120 degrees with each other 
(but not necessarily passing through a single point). 
\end{theorem}

\begin{proof}
By Theorem~\ref{th:recall}, see Corollary~\ref{cor:bnd}, 
we know that the regions with finite area are bounded.
By Theorem~\ref{th:recall2} we known that the boundary of the partition is a locally finite 
planar graph.
Some of the arcs of this graph might be unbounded, in that case we imagine 
that the arc has one or two vertices of order $1$ at infinity 
(which means that different unbounded 
arcs have different unbounded vertices at infinity).
All other vertices have order $3$ because the regularity of the boundary 
in the planar case states that exactly three edges can meet at a vertex point 
with equal angles of 120 degrees.
Since the regions with finite measure are bounded we can find a large radius $R>0$
such that all the bounded regions are compactly contained in $B_R$.
Outside this ball the arcs of the graph $\partial \vec E$ have zero curvature 
because we do not have any local constraint on the area enclosed by infinite 
regions. So, outside $B_R$, the graph $\partial \vec E$ is composed by straight 
lines (with two end-points at infinity), 
lines segments (with two end-points in $\RR^2$), or half-lines (with one  
end-point in $\RR^2$ and one end-point at infinity).

We claim that every bounded closed (hence finite) loop contained in 
$\partial \vec E$ is contained in $B_R$.
In fact take any bounded loop $\gamma$ and suppose that there is an arc 
$\alpha$ not completely contained in $B_R$. 
The two regions separated by this arc have both infinite area, 
because the regions with finite measure are all contained in $B_R$. 
So that $\alpha$ has to be a straight line segment adjacent to two connected components,  
each of just one among the two infinite-area regions. 
One of the two infinite regions separated by $\alpha$ has a connected 
component $C$ which is in the interior of the loop $\gamma$ and is adjacent to 
the arc $\alpha$.
If we remove $\alpha$, and reassign this component  
$C$ to the other infinite-area region we strictly decrease the perimeter, 
while preserving the area constraints,  
because we are exchanging a finite area between two regions with infinite area.
  This is also a variation with compact support since $\gamma$ is bounded.
  Hence we obtain a contradiction with the local isoperimetricity of $\vec E$.

  Now we claim that the graph $\partial \vec E$ has a finite number of vertices
  (and hence a finite number of edges since every vertex has finite order).
  Recall that all vertices of the graph have order 3 
  apart from the vertices at infinity which have order 1 by convention.
  We will call \emph{bounded vertices} the vertices which are not at infinity.
  The estimate 
  $P(\vec E, B_\rho\setminus \bar B_R)\le P(\vec E,B_\rho)\le C_0(2,N)\cdot \rho$
  (given by Lemma~\ref{lm:stima1})
  implies that the number of vertices at infinity is not larger than
  $C_0$ because each vertex at infinity is the end-point of an half-line 
  which asymptotically gives a contribution of $\rho$ to the perimeter 
  in the ball $B_\rho$. 
  So the graph has a finite number of vertices of order one at infinity, 
  in particular there is only a finite number of parallel entire straight lines.

  Suppose now, by contradiction, that we have an infinite number of bounded
  vertices, which have order three. 
  Since there are only a finite number of arcs that are entire  
  straight lines enlarging $R$, we can suppose that outside $B_R$  
  there are no entire straight lines.
  Since the graph is locally finite, we must have at least a sequence  
  of bounded vertices going to infinity.  
  Since the loops of $\partial E$ are all contained in $B_R$
  and since the graph $\partial E$ is locally finite, we have a finite number 
  of closed loops.
  By removing a finite number of arcs in $B_R$ we obtain a subgraph $\Gamma$
  without cycles, which is composed by a finite number of trees, each one touching 
  $B_R$.
  So in $\Gamma$ (and hence in $\partial E$) it is possible to find a
  tree with infinitely many vertices of order three.
  This tree contains hence infinitely many disjoint paths each 
  composed by infinitely many arcs: each such path must go to infinity 
  because the graph is locally finite. 
  But this, again, is in contradiction with the estimate 
  $P(\vec E,B_\rho)\le C_0\cdot \rho$.

  Since the graph $\partial \vec E$ is finite, by further enlarging $R$ 
  we might suppose that $B_R$ contains all the bounded vertices of the graph 
  so that $\vec E\setminus B_R$ 
  is composed by a finite number, 
  let say $n$, of disjoint half-lines emitted by $B_R$ and going to infinity. 
  These half lines, have all different directions towards infinity, because 
  if we had two parallel half-lines with the same direction we could 
  easily merge them into a single long segment and then split it again 
  to obtain a partition with smaller perimeter (see Figure~\ref{fig:parallel}).  
  
  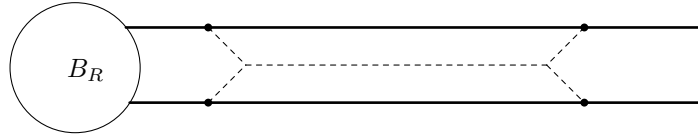
\begin{figure}
  \begin{tikzpicture}[line cap=round,line join=round,>=triangle 45,x=0.5cm,y=0.5cm]
    \clip(-5.4,-0.85) rectangle (13,2.7);
  \draw [dash pattern=on 2pt off 2pt] (1,1)-- (9,1);
  \draw [dash pattern=on 2pt off 2pt] (0,0)-- (1,1);
  \draw [dash pattern=on 2pt off 2pt] (0,2)-- (1,1);
  \draw [dash pattern=on 2pt off 2pt] (10,2)-- (9,1);
  \draw [dash pattern=on 2pt off 2pt] (9,1)-- (10,0);
  \draw(-3.54,0.93) circle (1.73);
  \draw [line width=1pt, domain=-2.190177598854443:13.26510937263622] plot(\x,{(--24.38-0*\x)/12.19});
  \draw [line width=1pt, domain=-2.0900355714923435:13.26510937263622] plot(\x,{(-0-0*\x)/12.09});
  \fill (0,0) circle (1.5pt);
  \fill (10,0) circle (1.5pt);
  \fill (0,2) circle (1.5pt);
  \fill (10,2) circle (1.5pt);
  \draw (-4,1.4) node[anchor=north west] {$B_R$};
  \end{tikzpicture}

\caption{The modification performed in the proof of Theorem~\ref{th:boundedness}
when two rays have the same direction. }
\label{fig:parallel}
\end{figure}

  Now we can consider a \emph{blow-down} of $\vec E$.
  By letting $\vec E^k = \frac{\vec E}{k}$ be a rescaled partition of $\vec E$ 
  we easily notice that $\vec E^k$ converges in $L^1_\loc$ 
  as $k\to+\infty$
  to a partition $\vec E^\infty$ of $\RR^2$ delimited by $n$ half-lines 
  with a common vertex at the origin, 
  each line parallel to the $n$-half-lines 
  of $\vec E\setminus B_R$.
  Since each $\vec E^k$ is a locally $J$-isoperimetric partition,
  with $J$ being the set of indices of the finite regions of $\vec E$,
  by Theorem~\ref{th:closure} we deduce that also $\vec E^\infty$ 
  is a locally $J$-isoperimetric partition.
  However for $j\in J$ the regions $E^k_j = \frac{E_j}{k}$
  converge to the empty set because $E_j$ is bounded. 
  So, we can remove the regions $E_j^\infty$ from $\vec E^\infty$ and obtain 
  a locally isoperimetric partition of $\RR^2$ composed by $n$ angles 
  with a common vertex in the origin.
  These angles cannot be smaller than 120 degrees (by the general regularity results 
  or by simple geometric considerations) hence $n\le 3$.
  In the case $n=0$ the graph $\partial \vec E$ has no vertices at infinity and  
  hence the partition $\vec E$ is a cluster.
  In the case $n=3$ we have that $\partial \vec E\setminus \bar B_R$ 
  is made of three half-lines going to infinity with 
  relative angles of 120 degrees. 

  In the case $n=2$ we have that $\partial \vec E\setminus \bar B_R$ is made 
  of two half-lines going to infinity at opposite 
  directions. In this case we claim that the two lines are collinear 
  (i.e.{} are contained in the same straight line). 
  
In fact if they were not collinear we could rotate the bounded cluster to 
  which they are attached, 
  creating two angles in the half-lines and strictly decreasing the perimeter
 (see Figure~\ref{fig:rotation}).
  
\begin{figure}
  \begin{tikzpicture}[line cap=round,line join=round,>=triangle 45,x=1.0cm,y=1.0cm]
  \clip(-4.30,-2.02) rectangle (4.30,2.02);
  \draw(0,0) circle (1cm);
  \draw [line width=1.2pt,domain=0.5905099943842171:4.452164861944947] plot(\x,{(--3.41--0.01*\x)/4.23});
  \draw [line width=1.2pt,domain=-5.120681908947257:-0.5905099943842171] plot(\x,{(--1-0*\x)/-1.24});
  \draw (1.83,0.81)-- (0.915,0.405);
  \draw (-0.915,-0.405)-- (-1.83,-0.81);
  \begin{scriptsize}
  \fill (0.59,0.81) circle (1.5pt);
  \fill (1.83,0.81) circle (1.5pt);
  \fill (-0.59,-0.81) circle (1.5pt);
  \fill (-1.83,-0.81) circle (1.5pt);
  \fill (0.91,0.41) circle (1.5pt);
  \fill (-0.91,-0.41) circle (1.5pt);
  \end{scriptsize}
  \end{tikzpicture}
  \caption{The rotation performed in the proof of Theorem~\ref{th:boundedness}.
  We can suppose that the two non collinear half-lines are emitted at 
  diametrically opposite points of $B_R$.}
\label{fig:rotation}
\end{figure}
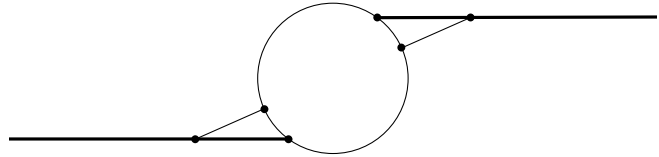

\end{proof}

\begin{theorem}[existence]\label{th:existence}
  Let $m_k\in [0,+\infty]$ for $k=1,\dots, N$ be a given 
  $N$-uple of areas such that at least one and 
  at most three of the $m_k$ are infinite.
  Then there exists an isoperimetric partition 
  $\vec E = (E_1, \dots, E_{N})$
  of $\RR^2$ whose regions have the prescribed measures.
  
  If all the areas are finite or at least four 
  of them are infinite then there are no isoperimetric partitions
  with the prescribed measures.
\end{theorem}

\begin{proof}
Theorem~\ref{th:boundedness} guarantees that 
in a locally isoperimetric partition there are at least one and at most
three infinite areas, 
so the second part of the statement has already been proved.

Let $M$ be the number of infinite areas, 
$1\le M\le 3$. 
Without loss of generality suppose that 
the infinite areas are the first $M$: 
$m_1$ if $M=1$, $m_1,m_2$ if $M=2$ and 
$m_1,m_2,m_3$ if $M=3$ while 
$m_k<+\infty$ for $k=M+1,\dots, N$.

If $M=1$ there exists an isoperimetric 
$(N-1)$-cluster $(E_2,\dots,E_{N})$ 
with the prescribed finite measures $m_2,\dots, m_N$. 
By adding the external region $E_1=\RR^2\setminus\bigcup_{k=2}^N E_k$ 
we obtain 
a locally isoperimetric partition with the given measures
$|E_k|=m_k$ for $k=1,\dots, N$.

We now consider the cases $M=2$ and $M=3$. 
If $M=2$ we let $\vec C = (C_1,C_2,\emptyset, \dots,\emptyset)$
be an $N$-partition such that $\partial \vec C$ is a straight 
line passing through the origin. 
If $M=3$ we let $\vec C = (C_1,C_2,C_3,\emptyset, \dots,\emptyset)$
be an $N$-partition such that 
$\partial \vec C$ 
is the union of three half-lines emitted by the origin 
with equal angles of 120 degrees. 

In both cases we consider a radius $R>0$ so large that 
$\abs{B_R}> m_{M+1}+\dots+m_{N}$. 
We then consider the family of all partitions $\vec E$ of 
$\RR^2$ which coincide with $\vec C$ outside of $\overline{B_R}$ 
and such that $\abs{E_k}=m_k$ for $k=M+1,\dots,N$.
In this family we minimize $P(\vec E,\overline{B_R})$ 
(thus taking into account also the length of $\partial B_R\cap \partial \vec E$).
By compactness and semicontinuity we know that a minimizer 
$\vec E$ always exists.
The minimizer, restricted to $B_R$, is a $J$-isoperimetric partition of $B_R$ 
with $J=\ENCLOSE{M+1,\dots, N}$ (see Definition~\ref{def:mixed_isoperimetric}).


Now we are going to complete the proof in the case $M=3$.
The case $M=2$ is similar but simpler so we don't treat it here.

\emph{Step 1: obtain an estimate on the perimeter of the sub-clusters.}
Let us define $D$ to be the union of all the bounded connected components 
of the regions $E_1,\dots,E_N$. 
Since the bounded components are all contained in $\bar B_R$ we have that 
$D\subset \bar B_R$ and since $E_4,\dots,E_N$ are bounded 
we have $D\supset E_4\cup \dots \cup E_N$. 
Moreover $D$ will also contain the 
bounded connected components of $E_1,E_2,E_3$, if they exist.
For $j=1,2,3$ let $E_j'=E_j \setminus D$ be the only unbounded connected
component of $E_j$ and 
consider the following partitions of $\RR^2$:
\begin{align*}
  \vec F^1 = (E_1'\cup D,E_2',E_3'), 
  \qquad 
  \vec F^2 = (E_1',E_2'\cup D,E_3'),
  \qquad
  \vec F^3 = (E_1',E_2',E_3'\cup D),
\end{align*}
and 
\[
  \vec G = (E_1',E_2',E_3',D).  
\]
Let $\partial \vec C \cap \partial B_R = \ENCLOSE{p_1, p_2, p_3}$ 
be the three fixed points on $\partial B_R$ enumerated so that
the half-line terminating in $p_j$ is non contained 
in $\partial E_j$, for $j=1,2,3$.
Notice that 
\[
  P(\vec F^1, \bar B_R) 
  = \H^1((\partial E_2' \cup \partial E_3')\cap \bar B_R).
\]
Since $E_2'$ is simply connected we know that 
$\partial E_2' \cap \bar B_R$ is a compact and connected set 
containing the two point $p_1$ and $p_3$.
Analogously $\partial E_3'\cap \bar B_R$
is a compact connected set containing $p_1$ and $p_2$.
Hence $(\partial E_2'\cup \partial E_3')\cap \bar B_R$
is a compact connected set containing $\ENCLOSE{p_1,p_2,p_3}$.
We know that the shortest compact connected 
set containing $\ENCLOSE{p_1,p_2,p_3}$ is the classical 
Steiner tree on the three vertices, which is known to have 
length $3R$ (see, for example, \cite{PaoSte12}). 
The same reasoning can be applied to the other two partitions
so we have 
\[
  P(\vec F^1,\bar B_R) + P(\vec F^2,\bar B_R) + P(\vec F^3, \bar B_R) 
    \ge 9R.
\]
On the other hand we have 
\[
  P(\vec F^1, \bar B_R) 
  = \H^1(\partial E_1' \cap \partial E_2')
  + \H^1(\partial E_1' \cap \partial E_3')
  + \H^1(\partial E_2' \cap \partial E_3')
  + \H^1(\partial D)
  - \H^1(\partial E_1' \cap \partial D)
\]
and analogously for $\vec F^2$ and $\vec F^3$.
Summing up we obtain 
\begin{align*}
 9R & \le 3 \Enclose{
    \H^1(\partial E_1' \cap \partial E_2')
  + \H^1(\partial E_1' \cap \partial E_3')
  + \H^1(\partial E_2' \cap \partial E_3')}
  + 2 P(D)\\
  & = 3 P(\vec G, \bar B_R) - 2 P(D)
\end{align*}
And since $P(\vec G,\bar B_R) \le P(\vec E,\bar B_R)$
we obtain 
\[
   P(D) \le 3 P(\vec E,\bar B_R)- 9R.
\]
Now it is not diffucult to estimate $P(\vec E,\bar B_R)$ 
by taking a competitor of $\vec E$ composed by the triple 
junction $\vec C$ with $N-M$ balls with given areas.
This implies that 
\[
  P(\vec E, \bar B_R) 
  \le 3 R + 2\pi\sum_{k=4}^{N} \sqrt{\frac{m_k}{\pi}}
\]
In conclusion we have found
\[
  P(D) \le d := 6\sqrt \pi\sum_{k=4}^{N} \sqrt{m_k}.
\]

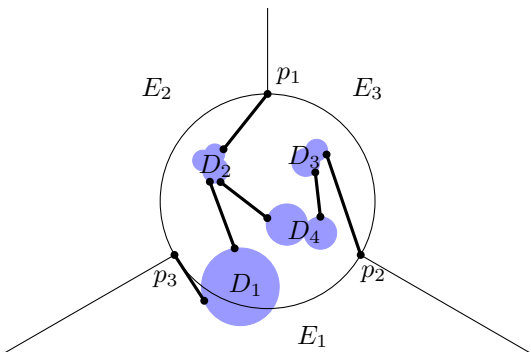
\begin{figure}
\definecolor{xdxdff}{rgb}{0.49,0.49,1}
\definecolor{uququq}{rgb}{0.25,0.25,0.25}
\definecolor{zzzzff}{rgb}{0.6,0.6,1}
\definecolor{qqqqff}{rgb}{0,0,1}
\begin{tikzpicture}[line cap=round,line join=round,>=triangle 45,x=0.2cm,y=0.2cm]
\clip(-27.2,-10) rectangle (23.52,12.85);
\draw [color=zzzzff,fill=zzzzff,fill opacity=1.0] (2.54,2.58) circle (0.92);
\draw [color=zzzzff,fill=zzzzff,fill opacity=1.0] (3.26,3.4) circle (0.72);
\draw [color=zzzzff,fill=zzzzff,fill opacity=1.0] (-3.47,2.06) circle (0.85);
\draw [color=zzzzff,fill=zzzzff,fill opacity=1.0] (-4.29,2.71) circle (0.69);
\draw [color=zzzzff,fill=zzzzff,fill opacity=1.0] (-3.51,3.15) circle (0.66);
\draw [color=zzzzff,fill=zzzzff,fill opacity=1.0] (1.28,-1.55) circle (1.37);
\draw [color=zzzzff,fill=zzzzff,fill opacity=1.0] (3.51,-2.11) circle (1.08);
\draw [color=zzzzff,fill=zzzzff,fill opacity=1.0] (-1.81,-5.69) circle (2.58);
\draw(0,0) circle (7.14);
\draw (0,7.14) -- (0,12.85);
\draw [domain=6.186674063787027:23.524063979025065] plot(\x,{(-0-5.29*\x)/9.15});
\draw [domain=-27.204873044599346:-6.186674063787026] plot(\x,{(-0-3.93*\x)/-6.81});
\draw (0,7.14) node[anchor=south west] {$p_1$};
\draw (5.6,-6) node[anchor=south west] {$p_2$};
\draw (-5.3,-4.0) node[anchor=north east] {$p_3$};
\draw (-9,6.14) node[anchor=south west] {$E_2$};
\draw (5,6.14) node[anchor=south west] {$E_3$};
\draw (3,-9) node {$E_1$};

\draw [line width=1.2pt] (-2.93,3.46)-- (0,7.14);
\draw [line width=1.2pt] (-3.13,1.28)-- (-0.02,-1.12);
\draw [line width=1.2pt] (-3.84,1.3)-- (-2.19,-3.13);
\draw [line width=1.2pt] (-4.21,-6.62)-- (-6.19,-3.57);
\draw [line width=1.2pt] (3.51,-1.03)-- (3.17,1.92);
\draw [line width=1.2pt] (6.19,-3.57)-- (3.92,3.11);
\draw (-3.2,-4.2) node[anchor=north west] {$D_1$};
\draw (-5.2,3.7) node[anchor=north west] {$D_2$};
\draw (0.7,4.3) node[anchor=north west] {$D_3$};
\draw (0.7,-0.7) node[anchor=north west] {$D_4$};
\begin{scriptsize}
\fill (-0.02,-1.12) circle (1.5pt);
\fill (3.17,1.92) circle (1.5pt);
\fill (0,7.14) circle (1.5pt);
\fill (-6.19,-3.57) circle (1.5pt);
\fill (6.19,-3.57) circle (1.5pt);
\fill (-3.13,1.28) circle (1.5pt);
\fill (-2.93,3.46) circle (1.5pt);
\fill (-3.84,1.3) circle (1.5pt);
\fill (-2.19,-3.13) circle (1.5pt);
\fill (-4.21,-6.62) circle (1.5pt);
\fill (3.92,3.11) circle (1.5pt);
\fill (3.51,-1.03) circle (1.5pt);
\end{scriptsize}
\end{tikzpicture}
\caption{
The subclusters $D_1,\dots,D_4$ in the proof of Theorem~\ref{th:existence}
Step 2.}
\label{fig:subclusters}
\end{figure}

\emph{Step 2: prove that there is a single subcluster.}
We now consider the connected components $D_1,\dots,D_m$
of $\bar D$ and call them \emph{subclusters}. 
The arcs of $\partial E$ which are not contained in $\bar D$
are separating two of the infinite regions. 
Having no constraint on the area of these regions we conclude 
that such arcs are straight line segments.
The union of $D$ with these segments is a compact connected 
set containing the three points $p_1,p_2,p_3$. 
It is connected because otherwise the three external components 
$E_1'$, $E_2'$ and $E_3'$ would not be separated by $\partial E$.

We now consider a minimization problem where we take 
any possible rigid motion of each subcluster
(in the whole plane) and any possible union of straight line 
segments so that the union of the subclusters and the segments
is a compact connected set containing the three points $p_1,p_2,p_3$
(see Figure~\ref{fig:subclusters}).
Then consider the configuration that minimizes the total length of the 
segments with the constraint that different subclusters do not 
overlap (the intersection of two subclusters has zero measure).
The diameter of each subcluster is bounded by a constant $d<R$, as shown in Step 1,
hence we might assume that the subclusters are uniformly bounded.
Also, the total number of segments is bounded because it is not convenient 
to join two different components more than once.
The existence of a minimizer is hence guaranteed because we 
are dealing with a finite dimensional minimization problem 
of a continuous function over a compact set (the position and orientation 
of the subclusters with the non-overlapping constraint 
and the position of the extremal points of the segments with the connectedness 
constraint).

The topology of the minimizer can be described by a graph $\Gamma$ 
which has the segments as arcs. 
The vertices of $\Gamma$ are represented by the three points $p_1,p_2,p_3$,
the subclusters $D_1,\dots, D_m$ and eventually other additional points 
where three segments meet in their extremal points
(in a minimizer we can always suppose that when many segments meet in a point 
which is not a point of a subcluster, the point is an extreme 
of each segment and the segments meet in triples with equal angles).
If two different subclusters of this minimizer are touching each other,
we consider their union as a single vertex of the graph $\Gamma$.
If a subcluster is touching one of the three points $p_1,p_2,p_3$, we assume 
there is an arc (representing a segment of zero length) 
joining the point to the subcluster.

The graph $\Gamma$ is connected in view of the connectedness requirement 
in the minimization problem. 
The three fixed points $p_1,p_2,p_3$ are vertex of order one in $\Gamma$
(if one of them has order two we can imagine to add a segment of zero length
to normalize the order to one).
There cannot be other vertices of order one because if a subcluster has a single 
segment attached to it, we could translate the subcluster along the segment
to decrease the total length.
Also, the graph $\Gamma$ contains no cycle because if there were a cycle
we could remove one of the segments of the cycle and decrease the total length.
So $\Gamma$ is a tree.

Now notice that if a vertex of $\Gamma$ has order two, the vertex represents 
a subcluster with two segments attached to it. 
These two segments must be collinear because otherwise we could rotate
the subcluster around one of the two points of contact to decrease the length 
of the segment joining in the other point.
Notice that if $R$ is large enough, since each subcluster has diameter 
not larger than $d$, it is not possible that a subcluster touches two of 
the three points $p_1,p_2,p_3$, so one of the two segments connecting 
the subcluster has positive length.
Now, since the two segments are collinear, we can translate the subcluster along 
the common line containing them and eventually 
make the subcluster bump against another subcluster.
In this way a vertex of order two can be removed and we 
can suppose that the graph contains no vertices of order two.
If both the segments emitted by a subcluster of order two are directly 
connected to two of the three vertices $p_1,p_2,p_3$, one of the two 
vertices must have order two and moving the subcluster against this vertex 
it will become a vertex of order three.

Now a simple combinatorial argument shows that a tree with exactly three 
vertices of order one and no vertices of order two must be a tripod.
This means that in our minimizer there is a single subcluster with three 
segments attached to it, each one touching one of the three points $p_1,p_2,p_3$.
Moreover this subcluster has diameter not larger than $d$ (see Step 1).

We now claim that if $R$ is large enough the subcluster must be contained 
in $B_{R/2}$. 
In fact if the subcluster is not contained in $B_{R/2}$
it contains a point $q$ such that $\abs{q}>R$ and the total
length of the three segments would be not smaller than 
$\abs{q-p_1} + \abs{q-p_2} + \abs{q-p_3} - 3d$.
We know that the minimum of $\abs{q-p_1} + \abs{q-p_2} + \abs{q-p_3}$ 
is uniquely obtained for $q=0$ and is equal to $3R$ (this is a well known 
Steiner problem on three points). 
If we add the restriction $\abs{q}\ge R/2$ we obtain that the minimum
length must have some length $cR$ with $c>3$.
So if the subcluster is not contained in $B_{R/2}$ the total length 
of the segments would be at least $cR-3d$.
On the other if we move the subcluster to touch the origin we would 
have a length not larger than $3R$ and $3R<cR-3d$ if $R$ is large enough.

At this point our minimizer represents a partition $\vec E'$ 
which is a competitor to the original partition $\vec E$
and has not larger perimeter. 
So we can replace $\vec E$ with $\vec E'$
and assume that $\vec E$ has a single subcluster $D$ with three segments
joining it to the three points $p_1,p_2,p_3$. 

\emph{Step 3. Equi-boundedness of the subcluster.}
At this point we have a partition $\vec E'$ which has the same area constraints 
of the initial partition $\vec E$ and with a perimeter which is not larger.
We would like to prove that the single subcluster in $\vec E'$ is completely 
contained in $B_R$ so that $\vec E'$ is a minimizer of the same problem 
for which $\vec E$ was a minimizer.

Notice that the single subcluster $D'$ of $\vec E'$ is obtained by taking all the connected 
components of the set $D$ defined in Step 1 (which were the subclusters of $\vec E$)
and joining them together without changing areas and perimeter. 
So $P(D')=P(D) \le d$ where $d$ is defined in Step 1 and is independent from $R$.
Since $\partial D'$ is connected, we can state that $D$ has diameter smaller than $d$
and hence is contained in a ball of radius smaller than $d$.
If $R$ is large enough we can hence suppose that $D'$ is contained in $B_{R/2}$,
otherwise we could move it towards the origin and decrease the length 
of the three connecting segments.

This means, finally, that the resulting partition $\vec E''$ is a minimizer as 
well as $\vec E$ and, without loss of generality, we can replace $\vec E$ 
with $\vec E''$.

\emph{Step 4. Letting $R\to +\infty$.}
For each large $R$ we consider the minimizer given by Step 3 and call it 
$\vec E^R$ to make explicit the dependence on $R$.
We know that $\vec E^R$ has a single sub-cluster which is contained in a ball 
of radius $d$ centered in some point $q_R$ such that $\abs{q_R}+d<R$.

We claim that $\frac{\abs{q_R}}{R} \to 0$ as $R\to +\infty$.
In fact suppose by contradiction that there is a sequence 
or $R\to+\infty$ such that $\abs{q_R}\ge 4\sqrt{dR}$.
If we translate $D$ by $-q_R$ and replace the three
segments joining the inner cluster to the circle $\partial B_R$
with three segments joining towards the center of the circle,
we obtain a competitor cluster $\vec F^R$ with the same 
prescribed areas, and whose perimeter can be estimated
by
\[
  P(\vec F^R,\overline{B_R}) \le 3R + P(D) \le 3R + d.
\]
On the other hand since $(\partial \vec E^R)\setminus B_d(q^R)$ 
is composed by three line segments, if we consider the union 
$\Sigma_R$ of the three segments joining the three vertices on 
$\partial B_R$ with the point $q^R$, we have the estimate
\[
  P(\vec E^R,\overline{B_R})
  \ge \H^1(\Sigma_R)-3d,
\]
and, by Lemma~\ref{lm:steiner} below,
\[
  \H^1(\Sigma_R)
  \ge R \cdot \ell\enclose{\frac{\abs{q^R}}{R}}
  \ge R \cdot \ell\enclose{4{\sqrt {\frac d R}}}
  \ge 3R + 12d + o(1).
\]
This is in contradiction with the minimality of $\vec E^R$,
because we obtain, for sufficiently large $R$:
\[
   P(\vec E^R,\overline{B_R}) - P(\vec F^R)
   \ge \H^1(\Sigma_R)-3d - (3R + d)
   \ge 9d + o(1) \ge 0.
\]

For all $R$ we have proven that the translated regions 
$F^R_k = E^R_k - q_R$ are all contained in the same 
ball $B_d$. 
Moreover, by the previous claim, we notice that the 
translated balls $B_R(-q^R)$ invade the whole plane
since $\abs{q^R}$ is much smaller than $R$ on a suitable 
sequence $R\to +\infty$.
So there is a sequence $R_n \to \infty$ and measurable sets 
$F_4,\dots, F_{N}$ 
such that 
for all $k > 3$ one has $F^{R_n}_{k}\to F_k$ in $L^1$.
Also, up to a subsequence, each of the three rays emanating from the 
subcluster of $\vec E^{R_n}$ must converge to a ray which has the same 
direction of the corresponding ray of the reference cone $\vec C$.
These rays divide the complement of the union of the 
regions $F_4,\dots,F_{N}$ into three regions $F_1,F_2,F_3$ 
which are thus the $L^1_{\loc}$ limit of the corresponding regions 
$F_1^{R_k}$, $F_2^{R_k}$, $F_3^{R_k}$ as $k\to +\infty$.
So we have found a partition $\vec F=(F_1,\dots,F_{N})$ of $\RR^2$
which is the $L^1_{\loc}$ limit of the partitions $\vec F^{R_k}$ 
of the balls $B_{R^k}$.
Clearly $\abs{F_1} = \abs{F_2} = \abs{F_3} = +\infty$ 
and $\abs{F_k} = m_k$ for $k=M+1,\dots, N$.

The partitions $\vec E^{R_n}$ 
are all locally $J$-isoperimetric in 
$\Omega_n = B_{R_n}-q_{R_n}$ with $J=\ENCLOSE{4,\dots, N}$
and hence, by Theorem~\ref{th:closure},
we obtain that $\vec F$ is locally $J$-isoperimetric in $\RR^2$.
But this is equivalent to say that $\vec F$ 
is locally isoperimetric in $\RR^2$, which is our conclusion.
\end{proof}

\begin{lemma}\label{lm:steiner}
  Let $V_1,V_2,V_3$ be the three vertices of a regular 
  triangle inscribed in the circle $\partial B_1$ centered 
  in the origin of the plane $\RR^2$.
  Let $P$ be any point and let 
  $\Sigma_P = [P,V_1]\cup[P,V_2]\cup[P,V_3]$ 
  be the union of the three segments joining $P$ with the
  vertices of the triangle.
  Let
  \[
    \ell(\rho) = \inf\ENCLOSE{\H^1(\Sigma_P)\colon \abs{P}\ge \rho}.
  \]

  Then, for $\rho\to 0^+$ one has
  \[
    \ell(\rho) \ge 3 + \frac{3}{4}\rho^2 + o(\rho^2).
  \]
\end{lemma}
\begin{proof}
  By assumption we have $\abs{V_i}=1$, $V_1+V_2+V_3=0$.
Since the function $P\mapsto \H^1(\Sigma_P)$ is coercitive 
and convex on $\RR^2$, 
the infimum defining $\ell(\rho)$ is attained
in a point $P$ such that $\abs{P}=\rho$.
The proof is concluded by the following computation: 
\begin{align*}
  \H^1(\Sigma_P) 
  &= \sum_{i=1}^3 \abs{P-V_i}
  = \sum_{i=1}^3 \sqrt{1+\rho^2-2\langle P,V_i \rangle}\\
  &= \sqrt{1+\rho^2}
  \sum_{i=1}^3\sqrt{1-2\frac{\langle P, V_i\rangle}{1+\rho^2}} \\
  &= (1+\rho^2 + o(\rho^2))
  \sum_{i=1}^3\Enclose{1-\frac{\langle P, V_i\rangle}{1+\rho^2}
  -\frac{1}{4}\enclose{\frac{\langle P, V_i\rangle}{1+\rho^2}}^2 
  + o(\rho^2)}\\
  &= (1+\rho^2 + o(\rho^2))
  \Enclose{3-\frac{1}{4}\sum_{i=1}^3\enclose{\frac{\langle P,V_i\rangle}{1+\rho^2}}^2+o(\rho^2)}\\
  &\ge (1+\rho^2 + o(\rho^2))
  \Enclose{3-\frac{3}{4}\rho^2+o(\rho^2)}\\
  &= \enclose{1 + \rho^2 + o(\rho^2)}
  \cdot \Enclose{3 - \rho + o(\rho^2)}
  = 3 + \frac{3}{4} \rho^2 + o(\rho^2).
\end{align*}
\end{proof}

\begin{theorem}[uniqueness]\label{th:uniqueness}
  Let $\vec E$ be any locally isoperimetric $N$-partition of the plane $\RR^2$
  with $N\le 4$.
  Then $\vec E$ is standard.
  This means that in the planar case the standard partitions, 
  enumerated in section~\ref{sec:standard},
  are the only locally isoperimetric partitions with their given 
  areas.
\end{theorem}

\begin{proof}
For $N=1$ there is only one (trivial) partition $\vec E=(\RR^2)$.

For $N=2$ the cluster cannot have triple points hence the boundary 
is either a circle or a single straight line: both are standard.

Consider the case $N=3$. 
If we have a single region with infinite area then we have 
a standard double-bubble cluster.

If we have exactly two regions with infinite area then, by Theorem~\ref{th:boundedness}
we know that outside a large ball the interface is contained in a single line separating 
the two unbounded regions while the region with finite measure is bounded inside the ball. 
If we remove the two half-lines going to infinity we obtain what 
we called a subcluster in the proof of Theorem~\ref{th:boundedness}. 
The subcluster is composed by the connected components of the finite region of the partition
and possibly by other bounded connected components of the unbounded regions.

Remember that every connected component of an isoperimetric partition 
is simply connected (othersise the boundary of the partition would be disconnected).
Hence a bounded connected component is a curvilinear $n$-agon with all 
equal internal angles of 120 degrees. 
We know that each side of the $n$-agon touches a different connected component of some 
other region of the partition. 
In fact if two sides of the $n$-agon touch the same connected component
of another region then the components of the partition enclosed by the 
two touching components can be moved without changin perimeter and areas 
up to the point of touching the rest of the cluster with a quadruple point,
which is not possible (see \cite{Ble87}).

The two half lines going to infinity 
are adjacent to the same connected component 
$C_1$ of a region (say $E_1$) on one side and the same connected component 
$C_2$ of another region (say $E_2$) on the other side. 
The exterior arcs of the subcluster are adjacent to an unbounded component (either $C_1$ or $C_2$) 
on the outer side and hence must be adjacent to a component of $E_3$ in the inside, otherwise 
we could remove the arc decreasing the perimeter and without changing the area of $E_3$ which 
is the only area to be prescribed here. 
This means that the subcluster is composed by a single two-sided component of $E_3$ 
and the cluster is a lens cluster.

Suppose now that $N=3$ and we have three regions with infinite area.
Then, by Theorem~\ref{th:boundedness} we have three half-lines emanating from the 
subcluster. The three half lines separate three unbounded connected components, 
say $C_1$, $C_2$ and $C_3$, of the 
three regions $E_1$, $E_2$ and $E_3$.
We claim that in this case all the regions are connected. 
In fact suppose that a region, say $E_1$, is disconnected and take 
a bounded connected component $C$ of $E_1$.
If we give $C$ to any other component of the partition we strictly decrease 
the perimeter without changing any prescribed area, since they are all infinite.
Since all the regions are connected the subcluster must be empty and what 
we get is a partition whose boundary is composed by three half lines 
joining with equal angles at a triple point: this is the triple junction.

Suppose now that $N=4$. If we have only one region with infinite measure 
than the partition is a cluster and it is known that it must be a standard 
triple bubble. By Theorem~\ref{th:boundedness} we cannot have four regions 
with infinite measure.
If we have three regions with infinite area, Theorem~\ref{th:boundedness}
tells us that the boundary of the partition contains three half-lines emanating 
by a bounded subcluster. 
The three half lines separate three unbounded connected components $C_1$,
$C_2$ and $C_3$ of three regions, say respectively $E_1$, $E_2$ and $E_3$.
As in the case $N=3$ with two infinite regions, we can easily state 
that the external arcs of the subcluster must be adjacent to a component 
of the fourth region $E_4$ in the inside, otherwise by removing the arc 
we decrease the perimeter leaving the area of $E_4$ unchanged. 
So the subcluster is composed by a single, unique, component of $E_4$ which must 
be a triangular region. Since all angle are 120 degrees and the three half lines 
also define angles of 120 degrees one with each other, we conclude that the only 
possibility is for $E_4$ to be a Relaux triangle and the partition 
to be a Releaux partition.

The last case is $N=4$ with only two regions with infinite measure. In this 
case Theorem~\ref{th:boundedness} tells us that the boundary of the partition 
contains two collinear half lines emanating from a bounded subcluster. 
The rest of the proof is devoted to this case, which is much harder than the 
previous ones. 

We will use some of the ideas used by Lawlor in \cite{Law19}.
If $\vec E(t)$ is a one-parameter family of partitions
and $\vec E = \vec E(t_0)$ is a \emph{stationary} partition we have 
\begin{equation}\label{eq:49278}
   \Enclose{\frac{d}{dt} P(\vec E(t), B_R)}_{t=t_0} 
   = \sum_{k=1}^N p_k \Enclose{\frac{d}{dt} \abs{E_k}}_{t=t_0}
\end{equation}
where $p_k$ are the \emph{pressures} of the regions $E_k$ of $\vec E$.

Let $\vec E = \vec E(m_3,m_4)$ be a locally isoperimetric partition 
with measures $(+\infty,+\infty,m_3,m_4)$ and let $\vec F = \vec F(m_3,m_4)$ 
be the \emph{peanut} partition with the same measures.
By the considerations above, we know that there is a ball $B_R$ such that 
$\partial \vec E\setminus B_R$ is contained in a straight line. 
By translating and rescaling we can also assume that such a line is 
passing through the origin, which is the center of the ball. 
Do the same for $\vec F$.
Since we know that the $\vec F$ is locally isoperimetric (Theorem~\ref{cor:examples})
we have $P(\vec E,B_R) = P(\vec F,B_R)$.
Now define 
\[
\tilde P(\vec E) = P(E, B_R) - 2R 
\]
and notice that this definition does not depend on $R$ 
(if the ball has the stated requirements)
and that \eqref{eq:49278} holds true for $\tilde P$ as well
since $P(\vec E,B_R)$ differs from $\tilde P(\vec E)$ by a constant depending only on $R$.
If $\alpha$ is an arc separating $E_3$ from $E_1$ 
we know that the curvature of 
$\alpha$ is $p_3-p_1=p_3$ since $p_1=0$ (see Theorem~\ref{th:recall2}.
Consider a one family of clusters $\vec E(t)$ 
with $\vec E(0)=E_0$ and such that for $t$ varying in a neighbourhood of $t=0$ the arc $\alpha$ 
is replaced with an arc with curvature $p_3+t$ while the rest of the cluster is unchanged.
Let $\vec F(t)=\vec F(\abs{E_3(t)},m_4)$ be the peanut
partition with the same measures as $\vec E(t)$ so that $\vec F(0) = \vec F(m_3,m_4)$.
Let $q_3$ and $q_4$ be the pressures of the regions $F_3$ and $F_4$ of $\vec F$.
Since $\abs{E_k(t)}=\abs{F_k(t)}$ we have:
\begin{align}\label{eq:4776698}
  \Enclose{\frac{d}{dt} \tilde P(\vec F(t))}_{t=0}
  &= q_3 \Enclose{\frac{d}{dt} \abs{F_3(t)}}_{t=0}
  + q_4 \Enclose{\frac{d}{dt} \abs{F_4(t)}}_{t=0}\\
  &= q_3 \Enclose{\frac{d}{dt} \abs{E_3(t)}}_{t=0} 
  = \Enclose{\frac{d}{dt} \tilde P(\vec E(t))}_{t=0}.
\end{align}
Now notice that $\tilde P(\vec E(t)) \ge \tilde P(\vec F(t))$ for all $t$ 
because $\vec F(t)$ is locally isoperimetric. 
Moreover $\tilde P(\vec E(0)) = \tilde P(\vec F(0))$ because $\vec E$ is also 
isoperimetric. Hence $\tilde P(\vec E(t))-\tilde P(\vec F(t))$ has a local
minimum at $t=0$ and hence $\frac{d}{dt} \tilde P(\vec E(t))-\frac{d}{dt} \tilde P(\vec F(t))=0$
for $t=0$.
Hence from \eqref{eq:4776698} we deduce $q_3 = p_3$.
Repeating the same argument with an external arc of $E_4$ we obtain also $q_4=p_4$.
So the pressures of $\vec E$ coincide with the pressures of $\vec F$.
The \emph{internal} arcs, i.e.\ the arcs separating $E_3$ from $E_4$
have curvature $p_3-p_4=q_3-q_4$ and hence are also equal to the curvature
of the internal arcs of $\vec F$. 
Now notice that two triangles with angles of 120 degrees and with the same 
curvatures of the three sides are congruent. 
This means that the triangular connected components of $E_3$ and $E_4$
are congruent with the triangular regions $F_3$ and $F_4$ and in particular they have the same 
area.
Now remember that we know, from the previous discussion, 
that $\partial E$ is composed by a finite sequence of connected components 
alternating between the regions $E_3$ and $E_4$. 
The first and last components are triangular and the others are quadrilateral.
To fix the ideas let us suppose that the first triangular component $C$ is a component 
of the regione $E_3$. We conclude that $C$ is congruent to $F_3$. 
But since $\abs{E_3}=\abs{F_3}$ we notice that $C=E_3$ and $E_3$ has no other components.
This means that the last triangular component must be a component of $E_4$. 
And hence, reasoning as before, we conclude that it is congruent to $F_4$ and hence 
also $E_4$ has no other components.
We have hence found that $\vec E$ is congruent to $\vec F$ and hence $\vec E$ is standard.
\end{proof}

\section{Appendix}

We now give a path of statements leading to first regularity results for isoperimetric partitions, see Theroem ~\ref{th:recall},  and for their limits. Their proofs can be obtained with minor modifications from the corresponding for isoperimetric clusters. For reader convenience we refer to the
 monograph \cite[Ch. 29, 30]{Mag12}, and give only some outline to link the argument with the proofs there exposed.

\begin{theorem}[volume fixing variations]\label{th:volume_fixing_variations}
If $\vec F = (F_1,\dots, F_N)$, $N\geq 1$, is a partition of an open connected  set $B \subseteq \RR^d$, with $\abs{F_i\cap B}>0$ for all $i=1,\dots, N$, for every suitably choosen family of interface points, doubly linking each region of the partition with a fixed one with infinite volume,
there exist positive constants $\eps_1$, $C_1$,  
 $\eps_2$,  $\eta$, and an open bounded set $A\Subset B$ (a finite union of  open balls centered in the choosen interface points of $\vec E$ with radius $\eps_1$),
with the following property:
for every proper partition $\vec F^\prime = (F^\prime_1,\dots, F^\prime_N)$ of $B$ such that
$\sum \abs{(F_i\triangle F^\prime_i)\cap A} < \eps_2$ 
and every $\vec a\in V=:\{ (a_1,\dots, a_N)\in \RR^{N}\colon \abs{a_i} < \eta, \sum_i a_i = 0\}$,
there exists a $C^1$ function $\Phi\colon V\times B\to \RR^d$ with    $\Phi_{\vec a}\colon B\to B$ diffeomorphism 
of class $C^1$, such that
\begin{enumerate}
  \item $\ENCLOSE{x\in B\colon \Phi_{\vec a}(x) \neq x} \Subset A$;
  \item for all $i=0,\dots,N$ 
  \[
    \abs{\Phi_{\vec a}(F^\prime_i) \cap A} = \abs{F^\prime_i\cap A} + a_i;
  \] 
  \item given any $\H^{d-1}$-rectifiable set $\Sigma$ one has
  \[
    \abs{\H^{d-1}(\Phi_{\vec a}(\Sigma)) - \H^{d-1}(\Sigma)} 
    \le C_1 \H^{d-1}(\Sigma)\cdot \sum_{i=0}^{N} \abs{a_i}, 
  \]
\end{enumerate}
so that for every open bounded set $\Omega\subseteq B$ containg $\overline{A}$ one has 

\[\abs{P(\Phi_{\vec a} (\vec F^\prime), \Omega) -P(\vec F^\prime, \Omega)}
\le C_1\cdot P(\vec F^\prime, A) \sum_{i=0}^{N} \abs{a_i}
\]
\end{theorem}
\begin{proof}
  See \cite[Lemma 29.13, Theorem 29.14]{Mag12}.
\end{proof}

\begin{lemma} (Almgren's Lemma)\label{lm:almg} 
Let $\vec E$ a partition of $B$ an open connected set ($\vert B\cap E_j\vert >0$, for all $j$). Then
 there exist: a finite union of well separated disjoint open balls $A$,  compactly contained in $B$, constants $C$, $\eps >0$, $\eta >0$ depending on $\vec E$ and $B$ such that:  
 for each $\triangle\subseteq B$ open bounded disjoint from  $A$, for each partition $\vec E^\prime$ such that: $\abs{E^\prime_j \cap B}>0$, for all $j$, $\sum \abs{ E_j\triangle E^\prime_j\cap A}<\eps$, and 
each partition $\vec F$  of $B$, such that is a variation of $\vec E^\prime$ compactly contained in $\triangle $ and $\sum \vert| E_j^\prime \cap{\triangle}|- |F_j  \cap {\triangle}|\vert<\eta$, there exists a partition $\vec F^\prime$ of $B$, such that

\begin{enumerate}

\item   $ F_j \triangle  F^\prime_j $ is compactly contained in $A$, for all $j$,  
\item for any bounded open set $\Omega$ containing $A$ it holds $|F^\prime_j\cap \Omega|= |E^\prime_j\cap \Omega|$, for all $j$,
\item for any bounded open set $\Omega$ containing $A$ it holds\\
\centerline{ $|P(\vec F^\prime,\Omega) - P(\vec F,\Omega)|\le C\cdot P(\vec E^\prime, A)|\cdot 
\sum \left\vert \abs{F_j\cap \triangle} -\abs{E^\prime_j\cap \triangle}\right\vert$.}

\end{enumerate}

\noindent Hence if $\vec E^\prime $ is an isoperimetric partition for any bounded open set $\Omega\supseteq A$:

\[P(\vec E^\prime,\Omega)\le P(\vec F,\Omega)+ C\cdot P(\vec E^\prime, A)\cdot \sum \left\vert \abs{F_j\cap \triangle} -\abs{E^\prime_j\cap \triangle}\right\vert\]

\end{lemma}

\begin{proof} Is exactly the same proof of Lemma 29.16, Corollary 29.17 in \cite{Mag12}. 
\end{proof} 

\begin{remark}\label{rm:almg} Following \cite[Corollary 29.17]{Mag12}, one choose two  unions,  $A_1, A_2$, with  $dist(A_1, A_2)>0$, of the well separated balls with centers  two families of interface points of $\vec E$, as in Theorem 
\ref{th:volume_fixing_variations}. Put $\eta$ be the minimum among the relative $\eta_1$, $\eta_2$. Then 
  if $r_1<\frac{dist (A_1, A_2)}2$, $\omega_dr_1^d<\eta$ (depending only on $\vec E$), one can choose  as $\triangle$ any $B(x,r_1)$, $x\in \RR^d$, and $A$ one among $A_1$ and  $A_2$.
\end{remark}

With minor modifications Lemma 30.2 in \cite{Mag12} (Infiltration Lemma) still holds for locally isoperimtric partitions:

\begin{lemma} (Infiltration Lemma)\label{lm:inflt} 
For any partition $\vec E$ of $\RR^d$ consider $A_1$, $A_2$, $\eps (\vec E)$, $r_1(\vec E)$ as in Lemma \ref{lm:almg} and Remark  \ref{rm:almg}, and put $O=A_1\cup A_2$. Then 
for $d\geq 2$ there is $\eps_0(d)<\omega_d$, and   
for any $K\geq 1$  there exist $0<r_0 (\vec E , K)<\min\{1,r_1\}$, such that 
 {for every locally isoperimetric partition} $\vec E^\prime$ with

 \[P(\vec E^\prime, O)< K,\, \sum \vert  E_j\triangle  E_j^\prime \cap O\vert < \eps,\]  
 
\noindent  for each $x\in \RR^d$, $r<r_0$, 
 and for each $\Lambda\subseteq \{ 1, \dots , N\}$, if

\[\sum_{j\in \Lambda} \vert E^\prime_j \cap B(x,r)\vert < \eps_0 r^d\]

then 

\[\sum_{j\in \Lambda} \left\vert E^\prime_j \cap B\left(x,\frac r2\right)\right\vert =0\]

\end{lemma}

\begin{proof} Let $x\in\RR^d$. Modifing the proof of \cite[Lemma 30.2]{Mag12}, one applies Lemma \ref{lm:almg} and Remark \ref{rm:almg} to $\vec E^\prime$ with pivot $\vec E$, finding: $A$ among $A_1$, $A_2$, $C(\vec E)$, $\eps (\vec E)$, $\eta(\vec E)$,  $r_0(\vec E)<\min\left \{ \frac{dist (A_1, A_2)}2,\,  1,\, \left(\frac\eta \omega_d\right)^{\frac 1d}\right\}$, so that, for open bounded  $\Omega \supseteq A$,  $r< r_0$,  and for  $\vec F$ partition with $F_h\triangle E_h^\prime \Subset B_r(x)$, it holds good

\[P(\vec E^\prime,\Omega)\le P(\vec F,\Omega)+ C\cdot K\cdot \sum \left\vert \abs{F_j\cap B_r(x) } -\abs{E^\prime_j\cap B_r(x)}\right\vert .\]

Hence the proof is the same as reported in \cite[Lemma 30.2]{Mag12} to choose the competitor $\vec F$ to cancel the infiltration; at the end,   with the notation there used, to get the decay estimate (30.19) $m(s)^{1-\frac 1d}\le 6m^\prime (s)$,  one observes that suffices decreasing $r_0$ putting  $r_0< \frac 1{8 CK}$.
\end{proof}

Moreover using universal upper $(d-1)$-density estimate of the perimeter of a locally isoperimetric $N$-partition, for example given here in lemma \ref{lm:stima1}, one gets the local perimeter bound $K$ on any isoperimetric partions $\vec E^\prime$ depending only on $d$, $N$, $\vec E$ and $O$, so that it holds:

\begin{corollary}\label{cor:inflt} For $d\geq 2$ there is $\eps_0(d)<\omega_d$ such that  if $\vec E$ is a partition that is $L^1_{loc}$-limit of a sequence $\vec E^k$ of  locally isoperimetric partitions, then 
 there is  $r_0(\vec E)<1$, such that for all $x\in \RR$, $\Lambda\subseteq \{1, \dots , N\}$, $r<r_0$, if

\[\sum_{j\in \Lambda} \vert E_j \cap B(x,r)\vert \le \eps_0 r^d\]

then 

\[\sum_{j\in \Lambda} \left\vert E_j \cap B\left(x,\frac r2\right)\right\vert =0\]

\end{corollary} 

\begin{corollary} Such a limit  partion of locally isoperimetric ones can be considered with open regions.
\end{corollary}

This allow to extend the regularity of minimizing clusters to locally isoperimetric partitions (Theorem \ref{th:recall}, and Theorem 30.1 in \cite{Mag12}), and  get also the following density estimates (same proof of Lemma 30.6 of \cite{Mag12}):

\begin{lemma}  For $d\geq 2$ there exists \emph{positive} constants $c_0\le c_1<1$  and $c_2$, such that if $\vec E$ is a locally isoperimetric partition there exists $r_0>0$ such that for every region $E_j$ whenever $\rho<r$ and $x\in \partial E_j$
\begin{enumerate}
\item $c_0\omega_d \rho^d\le |E_j\cap B(x,\rho)|\le c_1 \omega_d \rho^d$
\item $c_2\omega_{d-1} \rho^{d-1} \le P(E_j, B(x, \rho))$.
\end{enumerate}

\end{lemma}

\begin{corollary}\label{cor:bnd} If $\vec E$ is a locally isoperimetric partition then each regions with finite measure is bounded. Hence the subcluster of the chambers with finite volume is bounded and with finite perimeter.
\end{corollary}
\begin{proof}
Fix $F$ a region of $\vec E$ with finite positive volume and $|F|> \varepsilon >0$, take a large ball $B=B(\vec 0, R)$ such that $0<|F\setminus B|\le \varepsilon$. On other hand if $F$ were unbounded then it would be that for every  $B^\prime\Supset B$ concentric ball  $|F\setminus B^\prime|>0$. Both  $|F|<\infty$ and $|F\setminus B^\prime|>0$ yield $\partial F\setminus B^\prime\not=\emptyset$. So that if $r_0$ is given as in the volume density estimate, and the radius of $B^\prime$ is greater than $R+r_0$, for $x\in \partial F\setminus {B^\prime}$ one has $B(x,\rho)\subset \RR^d\setminus B$ for each $\rho<r_0$. Summing up $c_0\omega_d \rho^d \le |F\cap B(x, \rho)|\le|F\setminus B|\le \varepsilon$, so $c_0\omega_d r^d \le \eps$ 
that can not hold for every $\varepsilon>0$.

\end{proof}

Thanks to the infiltration lemma with pivot partition $\vec E$, \ref{lm:inflt}, \ref{cor:inflt}, one has to observe that the density estimate are rather uniform, so that it holds:

\begin{lemma}  For $d\geq 2$ there are positive constants $c_0\le c_1<1$ and $ c_2$, depending only on $d$,  and for every $\vec E$ $N$-partition $L^1_{loc}$ limits of locally isoperimetric partitions, 
there exists $0<r_0(\vec E)<1$ such that 
whenever $\rho<r_0$ and $x\in \partial E_j$, $1\le j \le N$

\item $c_0\omega_d \rho^d\le |E_j\cap B(x,\rho)|\le c_1 \omega_d \rho^d$
\item $c_2\omega_{d-1} \rho^{d-1} \le P(E_j, B(x, \rho))$

\end{lemma}

\begin{remark} Notice that if $\vec E^k\to \vec E$, in $L^1_{loc}$, are locally isoperimetric $N$-partitions, then with the same constants, whenever $\rho <r_0$, and $k$ is large enough, for any $1\le j \le N$ and $x\in \partial E^k_j$

\item $c_0\omega_d \rho^d\le |E^k_j\cap B(x,\rho)|\le c_1 \omega_d \rho^d$
\item $c_2\omega_{d-1} \rho^{d-1} \le P(E^k_j, B(x, \rho))$

\end{remark}
%
%
%
%

Similarly for partitions that are limit of isoperimetric partions we have that the regions with finite volume are bounded.

\begin{corollary} If $\vec E$ is a  partition that is $L^1_{loc}$ limit of isoperimetric partitions $\vec E^k$, then its region with finite volume are bounded.
\end{corollary}

\bibliographystyle{amsplain}

\providecommand{\bysame}{\leavevmode\hbox to3em{\hrulefill}\thinspace}
\providecommand{\MR}{\relax\ifhmode\unskip\space\fi MR }
\providecommand{\MRhref}[2]{%
  \href{http://www.ams.org/mathscinet-getitem?mr=#1}{#2}
}
\providecommand{\href}[2]{#2}

\end{document}